\documentclass{amsart}

\long\def\eatit#1{}

\usepackage[colorlinks=true, linkcolor=black, anchorcolor= black, citecolor= black, filecolor= black, menucolor= black, urlcolor=red]{hyperref}

\newtheorem{thm}{Theorem}[subsection]
\newtheorem{thm*}{Theorem}
\newtheorem{prop}[thm]{Proposition}
\newtheorem{lem}[thm]{Lemma}
\newtheorem{cor}[thm]{Corollary}

\newtheorem{conj*}{Conjecture}

\newtheorem{prob*}{Problem}

\theoremstyle{definition}

\newtheorem{defn}[thm]{Definition}
\newtheorem{defn*}{Definition}

\newtheorem{example}[thm]{Example}
\newtheorem{remark}[thm]{Remark}

\newcommand{\pr}[1]{{{\bf P}^{#1}}}
\newcommand{\skipit}[1]{{}}
\newcommand{\OO}{{\mathcal O}}
\newcommand{\KK}{K}

\newcommand{\myarrow}[2]{\hbox to #1pt{\hfil$\to$\hfil}{\hskip-#1pt{\raise
10pt\hbox to#1pt{\hfil$\scriptscriptstyle #2$\hfil}}}}

\newtheorem{rem}[thm]{Remark}

\begin{document}
\title[On plane rational curves and the splitting of the tangent bundle]{On plane
rational curves and the splitting of the tangent bundle}

\author[A. Gimigliano]{Alessandro Gimigliano}
\address{Dipartimento di Matematica e CIRAM\\
Universit\`a di Bologna\\
40126 Bologna, Italy}
\email{gimiglia@dm.unibo.it}

\author[B. Harbourne]{Brian Harbourne}
\address{Department of Mathematics\\
University of Nebraska\\
Lincoln, NE 68588-0130 USA}
\email{bharbour@math.unl.edu}

\author[M. Id\`a]{Monica Id\`a}
\address{Dipartimento di Matematica\\
Universit\`a di Bologna\\
40126 Bologna, Italy}
\email{ida@dm.unibo.it}

\date{May 13, 2013}

\thanks{Acknowledgments: We thank GNSAGA, and
the University of Bologna, which supported
visits to Bologna by the second author, who also
thanks the NSA for supporting his research.}



\begin{abstract} Given an immersion $\varphi:\pr1\to\pr2$, we give new approaches
to determining the splitting of the pullback of the cotangent
bundle. We also give new bounds on the splitting type for
immersions which factor as $\varphi:\pr1\cong D\subset X\to\pr2$,
where $X\to\pr2$ is obtained by blowing up $r$ distinct points
$p_i\in\pr2$. As applications in the case that the points $p_i$
are generic, we give a complete determination of the splitting
types for such immersions when $r\leq 7$. The case that $D^2=-1$
is of particular interest. For $r\leq8$ generic points, it is
known that there are only finitely many inequivalent $\varphi$
with $D^2=-1$, and all of them have balanced splitting. However,
for $r=9$ generic points we show that there are infinitely many
inequivalent $\varphi$ with $D^2=-1$ having unbalanced splitting
(only two such examples were known previously). We show that these
new examples are related to a semi-adjoint formula which we
conjecture accounts for all occurrences of unbalanced splitting
when $D^2=-1$ in the case of $r=9$ generic points $p_i$.  In the
last section we apply such results to the study of the resolution
of fat point schemes.
\end{abstract}

\maketitle

\section{Introduction}\label{intro}
We work over an algebraically closed ground field $\KK$. We are
interested in algebraic immersions $\varphi:\pr1\to\pr n$, thus
$\varphi$ is a projective morphism which is generically injective
and generically smooth over its image. The fact that $\varphi$
need not be everywhere injective or smooth means that the image
$\varphi(\pr1)$ may have singularities. It is well-known that any
vector bundle on $\pr1$ splits as a direct sum of line bundles
\cite{refB, refG}. This applies in particular to the pullback
$\varphi^*T_{\pr n}$ of the tangent bundle. It turns out to be
more convenient, yet equivalent, for us to study the splitting of
the pullback $\varphi^*\Omega_{\pr n}(1)$ of the first twist of
the cotangent bundle. Thus we will focus on $\varphi^*\Omega_{\pr
n}(1)$; it is isomorphic to
$\OO_{\pr1}(-a_1)\oplus\cdots\oplus\OO_{\pr1}(-a_n)$ for some
integers $a_i$. By reordering if necessary we may assume $a_1\leq
a_2\leq\cdots\leq a_n$; we call $(a_1,\ldots,a_n)$ the
\emph{splitting type} of $\varphi^*\Omega_{\pr n}(1)$. Pulling the
Euler sequence
$$0\to \Omega_{\pr n}(1)\to \OO_{\pr n}^{\oplus n+1}\to \OO_{\pr n}(1)\to 0$$
back via $\varphi$, it follows that
$a_1+\cdots+a_n=d_C$, where $d_C$ is the degree of $C=\varphi(\pr 1)$.

The question arises as to what splitting types can occur.
Most of the work on this problem is moduli-theoretic: given putative splitting types,
one asks if there are any $\varphi$ with those invariants
and if so, what one can say about the space of all such $\varphi$, or about the generic $\varphi$, etc.
See for example \cite{refAs1, refAs2, refRm}.
When $\varphi$ is an embedding, one can
ask for the splitting type of the normal bundle of $\varphi(\pr1)$. This question has also
attracted attention; see for example \cite{{refHu}, refGS, refEV1, refEV2, refCl, refRn} among many others.
However, if $n=2$ this latter question is not of much interest, both because
the normal bundle is itself a line bundle and because $C$ must at most be either a line or a conic.
In contrast, there is still much that is not yet understood regarding
the splitting types of $\varphi^*\Omega_{\pr 2}(1)$.

When $n=2$, the splitting types have the form $(a_1,a_2)$ for integers
$0\leq a_1\leq a_2$ such that $a_1+a_2=d_C$.
We will denote $(a_1,a_2)$ by $(a_C,b_C)$ and refer to it as the splitting type of $C$,
and we will refer to $\gamma_C=b_C-a_C$ as the \emph{splitting gap}. When the gap is
at most 1 (i.e., when $\gamma_C$ is as small as parity considerations allow),
we will say that $C$ has \emph{balanced} splitting or is \emph{balanced}, and
we will say that $C$ is \emph{unbalanced} if the gap is more than 1.

The multiplicities of the singularities of $C$
heavily influence $\gamma_C$. For example, if $C$ has a point of multiplicity $m$, then
results of Ascenzi \cite{refAs1} show that
\begin{equation}\label{Ascenzibnds}
\min(m,d_C-m)\leq a_C\leq\min\Big(d_C-m,\Big\lfloor \frac{d_C}{2}\Big\rfloor\Big);
\end{equation}
see Lemma \ref{splitlem} and Proposition \ref{Ascenzi}. These bounds are
tightest when we use the largest possible value for $m$; i.e., when $m$ is the multiplicity $m_C$ of a
point of $C$ of maximum multiplicity. If $2m_C+1\geq d_C$, it follows from these bounds that
$a_C=\min(m_C,d_C-m_C)$ and hence $b_C=\max(m_C,d_C-m_C)$ and $\gamma_C=|2m_C-d_C|$.
This prompts us to make the following definition.

\begin{defn*}
A rational projective plane curve $C$ is \emph{Ascenzi} if $2m_C+1\geq d_C$.
\end{defn*}

Ascenzi curves exist. For example, it is easy to see that for each $d\geq 3$ there is a rational
projective plane curve $C$ of degree
$d_C=d$ with exactly one singular point, of multiplicity $m_C=d_C-1$.
It follows that each such $C$ is Ascenzi, and its splitting type is $(1,d_C-1)$.

The main problem which we study here can be stated as follows:

\begin{prob*}\label{compprob}
Given a subspace $V=\langle \varphi_0,\varphi_1,\varphi_2 \rangle$ of dimension
3 in $K[\pr1]_d$ which gives a linear series $g^2_d$ on $\pr1$  defining a morphism which is an isomorphism
on a nonempty open subset,  find
the splitting type $(a_C,b_C)$ for the rational curve $C\subset \pr2$
given by the $g^2_d$, and, when $C$ is not Ascenzi, determine conditions
on the singularities of $C$ which force the splitting to be unbalanced.
\end{prob*}

This problem is closely related to that of determining the syzygies of the
homogeneous ideal
$(\varphi_0, \varphi_1, \varphi_2)$, since, as is well-known (see Lemma \ref{b1}),
$a_C$ is the least degree of such a syzygy.
These syzygies are of independent interest; see for example
\cite{refISV}, which studies the loci of $V$'s inside
the Grassmaniann $G(3,K[\pr1]_d)$ with respect to their syzygies, and
determines the dimensions of the loci.

We give two additional computational solutions to the first part of Problem \ref{compprob}
by showing that $b_C$ can be computed in terms of the
saturation degree of the ideal $(\varphi_0, \varphi_1, \varphi_2)\subset\KK[\pr1]$ (see Theorem \ref{b2}), and
by showing how to determine $\gamma_C$ using the
computationally efficient method of moving lines (see Theorem \ref{splittingM}), which was originally
developed to compute implicit equations of curves when given a
parameterization \cite{refSGD, refSSQK}.

Note that for a general immersion $\gamma_C$, the singularities of
$C$ are nodes (i.e., $m_C\leq 2$) whose disposition
in $\pr2$ is almost never generic. Thus if $C$ is a general rational curve
of degree $d_C>5$, then $C$ cannot be Ascenzi, and thus the splitting gap is not
completely determined by \eqref{Ascenzibnds}. Nonetheless,
Ascenzi proved that the general rational curve $C$ is balanced \cite{refAs1}.

But what can one say if it is not $C$ which is general, but rather it is the points at which $C$ is singular which are general? Thus we propose to
study $\gamma_C$ for rational curves $C$ when the points at which $C$ is singular
are generic points of $\pr2$; i.e., given generic points $p_1,\ldots,p_r\in\pr2$,
we require that $C$ be smooth away from the points $p_i$, and that $C'$ be smooth, where $C'$ is
the proper transform of $C$ on the surface $X$ obtained as the blow up
$\pi:X\to \pr2$ of $\pr2$ at the points $p_i$
(so not only is $C$ smooth away from the points $p_i$, but $C$ does not have any
additional infinitely near singularities). The immersion $\varphi$ in this situation
factors as $\pr1=C'\subset X\to\pr2$, so that $\varphi(\pr1)=\pi(C')=C$. In general, given a smooth rational curve
$D$ on $X$, it is convenient to use $a_D$, $b_D$ and $\gamma_D$ with
the obvious meanings; i.e., $a_D=a_{\pi(D)}$ etc. Similarly, we will say that $D$ is
Ascenzi if $\pi(D)$ is. To simplify statements of our results, we will also say $D$
is Ascenzi if $\pi(D)$ is a point.
In these terms the problem we propose to study, which is still open, is:

\begin{prob*}\label{prob}
Given a blow up $\pi:X\to\pr2$ at $r$ generic points $p_i$,
determine $\gamma_D$ for smooth rational curves $D\subset X$.
\end{prob*}

Given a curve $C\subset\pr2$ and distinct points $p_i$, we will denote ${\rm mult}_{p_i}(C)$
by $m_i(C)$. Note if $m_i(C)= 0$, then $p_i\not\in C$, and if $m_i(C)= 1$, then $p_i\in C$ but
$C$ is smooth at $p_i$. Given the blow up $\pi:X\to\pr2$ at distinct points
$p_1,\ldots, p_r\in\pr2$ and a divisor $D$ on $X$, it is well known that the divisor class $[D]$
(i.e., the divisor modulo linear equivalence) can be written uniquely as
$[dL-m_1E_1-\cdots-m_rE_r]$, where $L$ is the pullback via $\pi$ to $X$ of a line,
and $E_i=\pi^{-1}(p_i)$. If $D\subset X$ is a smooth rational curve with $d>0$, then
$C=\pi(D)$ is also a rational curve, and we have $[D]=[d_CL-m_1(C)E_1-\cdots-m_r(C)E_r]$.
We will refer to the integer vector $(d_C,m_1(C),\ldots,m_r(C))$
as the \emph{numerical type} of $C$ (or, by extension, of $D$) with respect to the points $p_i$.

Thus for example, $(d,d-1)$ is an unbalanced Ascenzi type
(i.e., the numerical type of an unbalanced Ascenzi curve) for
every $d\geq 4$. Computer calculations suggest many types also arise for unbalanced
non-Ascenzi curves with generically situated singular points, but up to now only two
have been rigorously justified (see Example \ref{AEpairs} for these two).
In contrast, the following theorem is proved in \S\ref{7pts}:

\begin{thm*}\label{7ptthm}
Let $X$ be the blow up of $r$ generic points of $\pr2$.
Among numerical types of smooth rational curves $D\subset X$, the following holds:
\begin{itemize}
\item[(a)] for $r\le6$, $(10,4,4,4,4,4,4)$ is the unique non-Ascenzi type
and curves of this type have balanced splitting;
\item[(b)] for $r=7$ there are infinitely many non-Ascenzi types, and for all but finitely many of these
types the curves have unbalanced splitting.
\end{itemize}
\end{thm*}

Our results in \S\ref{7pts} completely solve Problem \ref{prob} for $r\leq7$ by
classifying the numerical types for all smooth rational curves $D\subset X$ for $r\leq 7$ generic points,
and by determining the splitting gaps of curves of each type.
\medskip  

For larger values of $r$, a natural special case of Problem \ref{prob}
is to consider exceptional curves; i.e., smooth rational curves $D\subset X$ with $D^2=-1$.
This case arises, for example, when
studying graded Betti numbers for minimal free resolutions of ideals of fat points supported at generic
points $p_i$ (see \cite{refF1,refF2, refGHI1} and also \S\ref{appls}), but this case
is of interest in its own right,
since the exceptional curves represent an extremal case of Problem \ref{prob}.
Indeed, if ${\rm char}(\KK)=0$, it is known \cite{refD1, refD2} for every $r$ that every smooth rational curve
$D\subset X$ satisfies $D^2\geq -1$.
This is only conjectural if $r>9$ when ${\rm char}(\KK)>0$, but it is true in all characteristics if $r\leq9$.
For if $r<9$, then $-K_X$ is ample, hence $D^2 \geq -1$ follows from the adjunction
formula, $D^2=2g_D-2-K_X\cdot D$, since $g_D=0$ for a smooth rational curve $D$.
If $r=9$, then $-K_X$ is merely nef, so this argument gives only $D^2\geq -2$,
but one can show that if $D^2=-2$, then $D$ reduces by a Cremona transformation
centered in the points $p_i$ to $L-E_1-E_2-E_3$, which contradicts the fact that the points
$p_i$ are generic. For an exposition of the conjectural status when $r>9$, see \cite{refH3}.

When $r<9$ it is known that there are only finitely many numerical types of
exceptional curves, and they all are Ascenzi (see \S\ref{9pts}).
Thus the first interesting case of Problem \ref{prob} for exceptional curves
is $r=9$, for which we have the following result (proved, as well as Theorem \ref{thm2} below, in \S\ref{9pts}):

\begin{thm*}\label{classificationthm}
If $X$ is the blow up of $r=9$ generic points of $\pr2$, then:
\begin{itemize}
\item[(a)] $X$ has only finitely many Ascenzi exceptional curves;
\item[(b)] up to the permutations of the multiplicities, the only numerical
type of an unbalanced Ascenzi exceptional curve is
$(4,3,1,1,1,1,1,1,1,1)$; but
\item[(c)] $X$ has infinitely many unbalanced non-Ascenzi exceptional curves.
\end{itemize}
\end{thm*}

Heretofore only one non-Ascenzi exceptional curve was proved
to have unbalanced splitting (this being the one of type $(8,3,3,3,3,3,3,3,1,1)$
\cite[Lemma 3.12(b)(ii)]{refFHH}; see Example \ref{AEpairs}) when $r=9$.
Computational experiments suggest that
$X$ also has infinitely many balanced exceptional curves when $r=9$. Proving that is still
an open problem, but it would follow (see Remark \ref{infnonAscbal})
if Conjecture \ref{9ptconj} which we state below is true.

Our proof of Theorem \ref{classificationthm}(c)
applies the following sufficient numerical criterion
for an exceptional curve with $r=9$ to have unbalanced splitting:

\begin{thm*}\label{thm2}
Let $E$ be an exceptional divisor on the blow up $X$ of $\pr2$ at $r=9$ generic points $p_i$.
If $d_E=E\cdot L$ is even and each $m_i=E\cdot E_i$ is odd, then $a_E\leq (d_E-2)/2$ and $\gamma_E\geq 2$.
\end{thm*}

The hypothesis that $d_E$ be even and each $m_i$ be odd is equivalent to
the existence of a divisor class $[A]$ on $X$ such that $2[A]=[E+K_X+L]$. The proof that
$a_E\leq (d_E-2)/2$ depends on showing that $A$ has nontrivial linear syzygies
and relating these to syzygies of the trace of $A$ on $E$. I.e., it depends on
showing that the kernel of
$\mu_A:H^0(\OO_X(A))\otimes H^0(\OO_X(L))\to H^0(\OO_X(A+L))$
is nontrivial, and relating it to the kernel of
$H^0( \OO_E(A))\otimes H^0(\OO_X(L))\to H^0(\OO_E(A+L))$.

The formula $[A]=[E+K_X+L]/2$, which we can paraphrase by saying
that $A$ is a semi-adjoint of $L+E$,
is suggestive of some deeper structure that so far remains mysterious, but
extensive computational evidence suggests that the existence of $A$ is both necessary and sufficient for
$C$ to be unbalanced. In fact, up to permutation of the entries,
there are 1054 numerical types of exceptional curves $E$ on the blow up $X$ of $\pr2$
at $r=9$ generic points such that the image of $E$ is a curve $C$ of degree at most 61
(the number 61 is an arbitrary choice but large enough to give us some confidence in testing
our conjectures). For all of these 1054
the splitting gap was computed to be at most 2 (according to computations using randomly chosen points in
place of generic points), with the gap being exactly 2 in precisely the 39 cases for which an $A$ occurs
with $2[A]=[E+K_X+L]$. We thus make the following conjecture:

\begin{conj*}\label{9ptconj}
Let $E$ be an exceptional divisor on the blow up $X$ of $\pr2$ at $r=9$ generic points $p_i$.
Then there is a divisor class $[A]$ with $2[A]=[E+K_X+L]$
if and only if $\gamma_E>1$, in which case
$\gamma_E=2$ and $a_E=(E\cdot L-2)/2$.
\end{conj*}
\noindent Proving the conjecture
would give a complete solution to Problem \ref{prob} for exceptional curves with
$r=9$. It would also allow one to determine
the number of generators in every degree but one in any minimal set of homogeneous generators for
any fat point ideal with support at up to 9 generic points of $\pr2$; see \S\ref{appls}.

Computational evidence suggests more is true. Conjecture \ref{9ptconj} is a special case of
the following more general conjecture which
relates the occurrence of unbalanced splittings to the existence of a certain divisor $A$,
but whereas Conjecture \ref{9ptconj} specifies the divisor $A$ precisely, it is not yet clear
how to find $A$ in the context of our more general conjecture.
To state the conjecture, we need the following definition:
\begin{defn*}
The \emph{linear excess} of a divisor $A$, written ${\rm le}(A)$, is the dimension of the kernel of $\mu_A$.
\end{defn*}

Note that if  ${\rm le}(A)=1$ then $h^0( \OO_X(A))\geq2$, and in particular $|A|$ is not empty.
\begin{conj*}\label{rptconj}
Let $E$ be an exceptional divisor on the blow up $X$ of $\pr2$ at $r$ generic points.
Then $a_E = min \{a\ |\ A\cdot E = a\}$, where the minimum is taken over all divisors $A$ such that $-K_X\cdot A=2$,
$h^1(\OO_X(A))=0$ and ${\rm le}(A)=1$. In particular, $E$ has unbalanced splitting if and only if
$A\cdot E <\lfloor\frac{E\cdot L}{2}\rfloor$ for some such divisor $A$.
\end{conj*}

In Section \ref{compgap} we describe explicit
computational methods for determining splitting invariants.
All of the computational methods, however, involve first having a parameterization $\varphi$.
In case $C$ is the image in $\pr2$ of a smooth rational curve
on a blow up $X$ of $\pr2$ at generic points $p_i$, we recall
in \S\ref{params} an
efficient way to obtain a parameterization by reducing
$C$ to a line via quadratic transformations (see also \cite[\S A.2.1]{refGHI1}).
In \S\ref{eulerosubsect} and \S\ref{mul} we study the problem from a $\pr 1$-centered point of view. We show how the splitting type is related to the saturation index of the homogeneous ideal $(\varphi_0, \varphi_1, \varphi_2)$ and we recover Ascenzi's result, Lemma \ref{splitlem}.

In Section \ref{sect3} we obtain our new bounds on the splitting invariants of smooth rational curves $D$
on surfaces $X$ obtained by blowing up distinct points $p_i$ of $\pr2$, which we apply to
Problem \ref{prob} to obtain our results for the case of $r\leq 9$ generic points.

Finally, in Section \ref{appls} we explain how our results can be applied
to the study of the graded Betti numbers of fat point subschemes of $\pr2$. In particular, we describe an infinite family of fat point schemes having generic Hilbert function and ``bad resolution".

\section{Computing the splitting gap}\label{compgap}

\subsection{Ascenzi's bounds}\label{Ascenzisubsect}
The cotangent bundle $ \Omega_{\pr 2}$ of the plane will be denoted simply by $\Omega$.

\par Let $C\subset\pr 2$ be a rational curve of degree $d$, with singularities at $p_1,\dots,p_r$,
and  multiplicity $m_{p_i}(C)=m_i$ at $p_i$. Consider the normalization $p: C' \to C$. In the following we write ${\OO}_{C'}(k)$ for the line bundle ${\OO}_{\pr 1}(k)$ of degree $k$ on $C'\cong \pr 1$. 
\par The Euler sequence on $\pr 2$
$$0 \to \Omega (1) \to {\OO}_{\pr 2} \otimes H^0({\OO}_{\pr 2}(1)) \to{\OO}_{\pr 2}(1) \to 0$$
is a sequence of vector bundles, hence the pullback through $p$ of its restriction to $C$ is still exact and gives a short exact sequence of vector bundles on $\pr 1$:
\begin{equation}\label{eqn1}
0 \to p^*\Omega (1)\to{\OO}_{C'}\otimes H^0({\OO}_{\pr 2}(1))  \to {\OO}_{C'}(d) \to 0.
\end{equation}

We have
$$p^*\Omega (1)\cong {\OO}_{C'}(-a_{C})\oplus{\OO}_{C'}(-b_{C})$$
with $0\leq a_{C}\leq b_{C}$ and $a_{C}+ b_{C}=d$.
Setting $V= H^0({\OO}_{\pr 2}(1)) $  we can rewrite \eqref{eqn1} as
\medskip \begin{equation}\label{eqnstar}
0 \to {\OO}_{C'}(-a_{C})\oplus{\OO}_{C'}(-b_{C}) \to{\OO}_{C'}\otimes V \to {\OO}_{C'}(d) \to 0.
\end{equation}

\medskip From here to the end of section 2.1, we focus on the case that the singularities of the curve are resolved after just one blow up. 
\par If $\pi:X\to\pr2$ is the morphism obtained by blowing up distinct points $p_i$, then as noted above
a basis for the divisor class group ${\rm Cl}(X)$ (of divisors modulo linear equivalence)
is given by the classes $[E_i]$ of the exceptional
divisors $E_i=\pi^{-1}(p_i)$ and the class $[L]$ of the pull back $L$ of a line in $\pr2$.
Given a curve $C\subset\pr 2$ of degree $d$, with singularities at $p_1,\dots,p_r$,
and  multiplicity $m_{p_i}(C)=m_i$
at $p_i$, the class $[C']$ of the strict transform $C'$ of $C$ is  $[dL-\sum m_iE_i]$.
\par So in our assumptions  $C=\pi (C')$ is the image of a smooth rational curve $C'\subset X$ whose
class is $[C']=[dL-\sum_im_iE_i]$.
\par We recall that $C'$ is an \emph{exceptional curve}
in $X$ if it is smooth and rational with $-1=(C')^2=d^2-\sum m_i^2$, which
by the adjunction formula implies $-1=K_X\cdot C'=-3d+\sum m_i$, since
$[K_X]=[-3L+E_1+\cdots+E_r]$.

Given a divisor $F$ on $X$, we will use $F$ to denote its divisor class and sometimes even the sheaf $\OO _X(F)$, and we will for convenience write $H^0(F)$ for $H^0(\OO _X(F))$. 

Since we can identify $V= H^0({\OO}_{\pr 2}(1))$ with  $H^0(L)$, we can rewrite \eqref{eqnstar} as
\begin{equation}\label{eqn5}
0 \to {\OO}_{C'}(-a_{C})\oplus{\OO}_{C'}(-b_{C}) \to{\OO}_{C'}\otimes H^0(L) \to {\OO}_{C'}(d) \to 0.
\end{equation}

In analogy with the Euler sequence, for each $i$ there is a bundle ${\mathcal M_i}$
giving a short exact sequence of bundles
$$0 \to {\mathcal M_i} \to {\OO}_X \otimes H^0(L-E_i) \to{\OO}_X(L-E_i) \to 0.$$
Restricting to $C'$ gives
$$0 \to {\mathcal M_i}|_{C'} \to {\OO}_{C'} \otimes H^0(L-E_i) \to{\OO}_{C'}(L-E_i) \to 0.$$
Using the injection of bundles $\OO_X(L-E_i)\to\OO_X(L)$ one can show that ${\mathcal M_i}$ is a subbundle
of $\pi^*\Omega (1)|_{C'}$ isomorphic to $\OO_{C'}(m_i-d)$, and
the sheaf quotient turns out to be isomorphic to the bundle ${\OO}_{C'}(-m_i)$.
Here the normalization morphism $p: C' \to C$ is $\pi \vert _{C'}$, so, given the isomorphism  $\pi^*\Omega (1)|_{C'}\cong {\OO}_{C'}(-a_C)\oplus{\OO}_{C'}(-b_C)$,
we have an exact sequence
$$0\to \OO_{C'}(m_i-d)\to {\OO}_{C'}(-a_C)\oplus{\OO}_{C'}(-b_C)\to {\OO}_{C'}(-m_i)\to 0$$
from which the following result of Ascenzi \cite{refAs1} is a direct consequence if we add the assumption that the singularities of the curve are resolved after just one blow up
(see the proof of Theorem 3.1 of \cite{refF1}
for details; also see \cite{refF2}).

\begin{lem}\label{splitlem} Let $C$ be a rational plane curve
of degree $d=d_C$ and assume that $C$ has a multiple point of  multiplicity $m$; let $a=a_C$, $b=b_C$. Then
we have $\min(m,d-m)\leq a\leq \min(d-m, \lfloor\frac{d}{2}\rfloor)$.
Thus if $d>2m+1$, then $m\leq a\leq \lfloor\frac{d}{2}\rfloor$,
while if $d\leq 2m+1$ (i.e., $C$ is Ascenzi), then the splitting type is completely determined:
if $d\leq 2m$ it is $(d-m,m)$, and if $d=2m+1$ it is $(m,d-m)$.
\end{lem}

\subsection{Splitting type, syzygies and the parameterization ideal}\label{eulerosubsect}

We denote the homogeneous coordinate ring of $\pr 1$ by $S=\KK[s,t]=\KK[\pr 1]$ and
that of $\pr2$ by $R=\KK[x_0,x_1,x_2]=\KK[\pr2]$.

Every rational curve $C \subset \pr 2$ can be defined parametrically by
homogeneous polynomials $\varphi_0,\varphi_1,\varphi_2\in S$
with no common factor and which give a $g^2_d$ series on $\pr 1$. They therefore define a
morphism $\varphi: \pr 1 \to \pr 2$ corresponding to the ring map

\begin{equation}\label{eqn2}
\begin{array}{cccc}
\widetilde \varphi : & R= \KK[x_0, x_1,x_2]& \to           & S=\KK[s,t]\\
&&&\\
                         & x_i                             & \mapsto &\varphi_i=\varphi _i(s,t).
\end{array}
\end{equation}
The kernel of this homomorphism is a principal ideal, generated by the implicit equation of
the curve $C$.

Assume as before that $C\subset\pr 2$ is a rational curve of degree $d$, with singularities at $p_1,\dots,p_r$, multiplicity $m_{p_i}(C)=m_i$ at $p_i$, and normalization $p: C' \to C$.
For notational simplicity, we set $a=a_C$, $b=b_C$.
Consider the sequence \eqref{eqnstar} twisted by $k-d$ for various $k\in \mathbb Z$:
$$0 \to {\OO}_{C'}(-a-d+k)\oplus{\OO}_{C'}(-b-d+k) \to
{\OO}_{C'}(-d+k)\otimes V \to {\OO}_{C'}(k) \to 0 \eqno(\star)_k$$
and search for the minimum $k\geq0$ such that $(\star)_k$ is exact on global sections. But
$$H^1({\OO}_{C'}(-a-d+k)\oplus{\OO}_{C'}(-b-d+k) )=0$$
if and only if $k\geq b+d-1$, so we have the following exact sequences
$$0 \to H^0({\OO}_{C'}(-a-d+k))\oplus H^0({\OO}_{C'}(-b-d+k)) \to H^0({\OO}_{C'}(-d+k))\otimes V
\xrightarrow{\psi_k} H^0({\OO}_{C'}(k) ) \eqno{(\star\star)_k}$$
with $\psi_k$ surjective for $k\geq b+d-1$. Note that we can identify $\bigoplus_kH^0({\OO}_{C'}(k) )$ with $S$.
Thus by taking the direct sum over all $k$, we obtain an exact sequence of graded $S$-modules.
With this in mind, we will write $S( \ell)_k$ in place of $H^0({\OO}_{C'}(\ell+k))$.

Now choose three linear forms $f_0,f_1,f_2$ which give a basis  of $V$.
Then an arbitrary element
$$\sum_{i=1,\dots,q} (h_i \otimes \sum_{j=0,1,2} c_{ij}f_j) \in S(-d)_k \otimes V$$
(where the $c_{ij}$ are constants) can be written as
$\sum_{j=0,1,2} \widetilde h_j \otimes f_j$ with $\widetilde h_j= \sum_{i=1,\dots,q} c_{ij}h_i $, and the map $\psi_k$ becomes
$$\begin{array}{cccc}
\psi _k: & S(-d)_k \otimes V                              & \to            & S_{k} \\
&&& \\
              & \sum_{j=0,1,2} \widetilde h_j \otimes f_j& \mapsto & \sum _{j=0,1,2} (\widetilde h_j)(f_j|_{C'})
\end{array}
$$
or, applying the natural identification of $S(-d)_k \otimes V$ with $S(-d)_k ^{\oplus 3}$,
$$\begin{array}{cccc}
\psi _k: & S(-d)_k ^{\oplus 3}& \to & S_{k} \\
&&& \\
 & ( \widetilde h_0, \widetilde h_1, \widetilde h_2)& \mapsto & \sum _{j=0,1,2} (\widetilde h_j)(f_j|_{C'})
 \end{array}
$$
Notice that in the identification of $C'$ with $\pr 1$, $|L|$
gives divisors of degree $d$ when restricted to $C'$, hence the $f_j|_{C'}$ are
forms of degree $d$ in the coordinate ring of $\pr 1$.
We usually choose $f_j=x_j$, $j=0,1,2$.

Taking direct sums of $(\star\star)_k$ for $k\geq b+d-1$ gives the following exact sequence of graded $S$-modules
$$0 \to(\oplus _{k\geq b+d-1} S(-a-d)_k)\oplus (\oplus _{k\geq b+d-1} S(-b-d)_k) \to \oplus _{k\geq b+d-1} S(-d)_k ^{\oplus 3} \xrightarrow{\oplus\psi_k}\oplus _{k\geq b+d-1} S_k\to 0$$
and by sheafifying we get back the exact sequence $(\star)_0$:
$$0 \to {\OO}_{C'}(-a-d)\oplus{\OO}_{C'}(-b-d) \to {\OO}_{C'}(-d)^{ \oplus3} \to {\OO}_{C'} \to 0.$$

Now assume the curve $C \subset \pr 2$ is given by parametric equations \eqref{eqn2},
and that the basis $f_0,f_1,f_2$ of $V$  we chose above is $x_0,x_1,x_2$.
Since the restriction of $x_j$ to $C$ is $\varphi_j$, we have $f_j|_{C'}=\varphi_j$ for $j=0,1,2$.

Notice that $C$ is a line if  and only if there is a degree zero relation $\sum c_j\varphi_j =0$
among  $\varphi_0,\varphi_1,\varphi_2$, with $c_j\in \KK$; that is, if and only if
$\varphi_0,\varphi_1,\varphi_2$ is not a minimal system of generators for the ideal
$J :=  (\varphi_0,\varphi_1,\varphi_2)$. For the rest of \S\ref{eulerosubsect} we assume $C$ is not a line,
i.e. that $d\geq 2$. Also notice that if we change the basis  $f_0,f_1,f_2$ of $V$
their restrictions to $C'$ still generate the same ideal $J$.

Regarding $J$ as a graded $S$-module, consider its minimal graded free resolution
\begin{equation}\label{eqn3}
0 \to  S(-c-d)\oplus S(-e-d)
\xrightarrow{\hbox{\tiny$\left(
\begin{matrix}
\alpha_0 & \beta _0  \\
\alpha_1 & \beta _1  \\
\alpha_2 & \beta _2
\end{matrix}
\right)$}}
S(-d)^{ \oplus3}
\xrightarrow{\hbox{\tiny$(\varphi_0\ \varphi_1\ \varphi_2)$}}
J\to 0,
\end{equation}
so we have $1\leq c\leq e$, $\deg \alpha_j=c$ and $\deg \beta_j=e$.
If we sheafify the sequence \eqref{eqn3}, we get the following exact sequence of ${\OO}_{C'}={\OO}_{\pr 1}$-modules:
\begin{equation}\label{eqn4}
0 \to {\OO}_{C'}(-c-d)\oplus{\OO}_{C'}(-e-d) \to
{\OO}_{C'}(-d)^{ \oplus3} \to {\OO}_{C'} \to 0.
\end{equation}
Since the zero scheme of the ideal $J$ is the empty set, hence
by the homogeneous Nullstellensatz the associated sheaf is ${\OO}_{C'}$.
Comparing this with $(\star)_0$, we see that $(c,e)=(a,b)$, i.e.:

\begin{lem}[{\cite[Lemma 1.1]{refAs2}}]\label{b1} Let $C$ be a rational plane curve of degree $d\geq 2$
and consider the pair $(a,b)$ with $1\leq a \leq b$ and $a+b=d$.
Then $(a,b)$ is the splitting type of $C$ if and only if $a$ is the minimal degree of a sygyzy of $J$.
\end{lem}

Alternatively, recall that the \emph{saturation} of a homogeneous ideal $J\subseteq(s,t)\subset S$
is the largest homogeneous ideal ${\rm sat}(J)\subseteq (s,t)$
such that for some $i\geq 1$, ${\rm sat}(J)\cap (s,t)^i=J\cap (s,t)^i$.
We call the least such $i$ the \emph{saturation degree} of $J$. If $i=1$, we say $J$ is \emph{saturated}.
For example, if $J\subseteq (s,t)$ has homogeneous generators with no non-constant common factor,
such as $J=(s^3,t^2)$, then $(s,t)={\rm sat}(J)$ by the homogeneous Nullstellensatz,
and the saturation degree of $J$ is the least degree $i$ such that $J_i=S_i$.
Thus the saturation degree in such a case can be computed from the Hilbert function of $J$
(i.e., from the function giving the dimension of $J_i$ as a function of $i$).
We now have:

\begin{thm}\label{b2} Assume $C \subset \pr 2$ is a rational curve of degree $d\geq 2$
with splitting type $(a,b)$, $a\leq b$, which
is given by parametric equations
\eqref{eqn2}. If $\sigma(\varphi)$ denotes the saturation degree of the ideal
$(\varphi_0 ,\varphi_1 , \varphi_2)$, then
$$b+d-1=\sigma(\varphi)\leq 2d-2.$$
\end{thm}

\begin{proof}
If $k\geq b+d-1$ the sequence \eqref{eqn3} in degree $k$ is the same as $(\star\star)_k$.
On the other hand, $(\star\star)_k$ is not exact on the right for $k\leq d+b-2$, so
$J_k=S_k$ if and only if $k\geq b+d-1$, hence $b+d-1=\sigma(\varphi)$.
Since $b\leq d-1$, we also have $\sigma(\varphi) \leq 2d-2$.
\end{proof}

\subsection{Parametric equations and multiple points}\label{mul}

Assume that our rational plane curve $C$ is given by parametric
equations as in \eqref{eqn2}, and has a multiple point $p$
of multiplicity $m$, where $p=[\ell_0,\ell_1,\ell_2]$. We can assume
$\ell_0=1$; we define $q(s,t)$ to be the greatest common factor of
$\varphi_1 (s,t)-\ell_1\varphi_0(s,t)$ and $\varphi_2
(s,t)-\ell_2\varphi_0(s,t)$. Hence there exist $h,g \in S_{d-m}$ such
that $\varphi_1 =\ell_1\varphi_0+qh$ and $\varphi_2
=\ell_2\varphi_0+qg$.

The generic line through $p$, $\alpha (x_1-\ell _1x_0)+\beta (x_2-\ell _2x_0)=0$, meets $C$ at $p$ with multiplicity $m$; i.e., the equation $q(\alpha h+\beta g)=0$ has $m$ roots counted with multiplicity corresponding to the point $p$. Hence the polynomial $q$ defines a divisor $n_1w_1+\dots + n_rw_r$ on $\pr 1$, with $\varphi (w_1)=\dots =\varphi (w_r)=p$ and $n_1+\dots+n_r=m$.

This means that for each point $p$ of multiplicity $m$ for
$C$ the ideal $J=(\varphi_0 ,\varphi_1 , \varphi_2)$ can be written
as $J=(\varphi_0, qh,qg)$ with $\deg(\varphi_0)=d$, $
\deg(q)=m$, $\deg(h)=\deg(g)=d-m$ for $q$, $g$, $h$
depending on the singular point $p$. This allows us to better
understand Theorem \ref{b2} and to recover Lemma \ref{splitlem}.

\begin{prop}\label{Ascenzi} Let $C \subset \pr 2$ be a rational plane curve of degree $d$ given
parametrically by $\varphi : \pr 1 \rightarrow \pr2$, as before. Let
$p\in C$ be a multiple point $p$ of multiplicity $m\geq 1$,
and let $(a,b)$ be the splitting type of $C$. Then $a\geq \min \{m,d-m\}$
and $b\leq \max \{m,d-m\}$, with
$(a,b)=(d-m,m)$ if $d\leq 2m$, and $(a,b)=(m,d-m)=(m,m+1)$ if $d=2m+1$.
\end{prop}

\begin{proof}  As above,  let $J=(\varphi_0,\varphi_1,\varphi_2)=(\varphi_0, qh,qg)$
where $\varphi_0, q,h,g \in S=\KK[s,t]$, $\deg(\varphi_0)=d$,
$\deg(q)=m$ and $\deg(h)=\deg(g)=d-m$. Since $h,g$
have no common factor, as well as $\varphi_0, qh,qg$, we have  $\dim
_K \langle h,g\rangle=2$, and $J_{d} =\langle\varphi_0\rangle\oplus
\; q\langle h,g\rangle$.
\medskip We now look at the multiplication
maps $$\nu _k: J_{d} \otimes S_{k} \to S_{d+k}.$$ Let $\bar k$ be
the least $k$ such that $\nu _{\bar k}$ is onto;  the saturation
degree $\sigma (\varphi)$ of $J$ is $d+\bar k$, and by  Theorem
\ref{b2}, $b=\bar k+1$.

We first prove that $\nu _{m-2}$ is never onto, so we conclude that $b \geq m$: we have $ Im(\nu)_{m-2} = \varphi _0 S_{m-2} +q(h,g)_{d-2}$ with $q(h,g)_{d-2} \subseteq qS_{d-2}$,
hence $\dim Im(\nu)_{m-2}\leq m-1+d-1= d+m-2 <\dim S_{d+m-2}$.
\medskip \par
Since both $(h,g)$ and $(\varphi_0, q)$ are regular sequences in
$S$, we have the minimal free resolutions for the ideals $(h,g)$ and
$(\varphi_0, q)$ of $S$:
$$
0\to S(-2d+2m)\to \oplus ^2 S(-d+m) \to (h,g)\to 0 \eqno{(\star')}
$$

$$
0\to S(-d-m)\to S(-d)\oplus  S(-m) \to (\varphi_0 , q)\to 0
\eqno{(\star'')}
$$

\medskip
Assume $d\leq 2m$; then $(\star')$ gives $\dim (h,g)_{d-1} = 2\dim S(m-1) - \dim
S(-d+2m-1) = 2m-(2m-d) = d$, so $(h,g)_{d-1}= S_{d-1}$, this
implying that $(h,g)_{k}= S_{k}$ for $k\geq d-1$.

Let us consider $\nu _{m-1}$. We have $J_{d} =\langle \varphi_0\rangle \oplus \; q\langle h,g\rangle$, so that $Im (\nu _{m-1})= \varphi_0
S_{m-1} + q(h,g)_{d-1}=\varphi_0 S_{m-1} + qS_{d-1}=(\varphi_0 ,
q)_{d+m-1}$. The sequence $(\star'')$ gives  $\dim (\varphi_0 ,
q)_{d+m-1} = m+d $, so that $\nu _{m-1}$ is surjective and $b=m$ in
this case. Since we are in the assumption $d-m\leq m$, in particular
we have $b\leq \max \{m,d-m\}$, and hence $a\geq \min \{m,d-m\}$.

Now let $d= 2m+u$, with $u\geq 1$. From the resolution $(\star')$ of $(h,g)$ we get
that $(h,g)_{d+u-1}=S_{d+u-1}$ and this implies, as in the previous
case, that $\nu_{d-m-1}$ is surjective, i.e. that $b\leq d-m$. Since
we are in the assumption $d-m>m$, we have $b\leq \max \{m,d-m\}$. In
particular, when $u=1$, this trivially implies that
$(a,b)=(m,d-m)=(m,m+1)$.
\end{proof}

\subsection{A  moving line algorithm for the splitting type}\label{CAD}

These kinds of questions are of interest also to people working in
Computer Aided Geometric Design (CAD). In fact one of the problems
they are interested in is how to compute the implicit function
defining a rational plane curve which is given by parametric
equations. This is a classical problem in algebraic geometry, traditionally solved
via resultants, but this gives rise to computing
determinants of rather large matrices, hence it is quite valuable to
find more efficient ways to get the implicit equation. One of the ways
this is done is by the method of ``moving lines"
\cite{refCSC,refSSQK,refSGD}. We will see that
this approach also offers algorithms with which we are able to deal with the
splitting problem.

\begin{defn}\label{0}
A \emph{moving line of degree $k$ for $C$} is an equation of the form
$\alpha_0(s,t)x_0+\alpha_1(s,t)x_1+\alpha_2(s,t)x_2=0$ where
$\alpha_i(s,t)\in K[s,t] _k$, such that
$\alpha_0(s,t)\varphi_0(s,t)+\alpha_1(s,t)\varphi_1(s,t)+\alpha_2(s,t)\varphi_2(s,t)$
is identically zero; hence a moving line  of degree $k$ is nothing
else than a family of lines parameterized by $\pr1$, giving a
syzygy of degree $k$ of the ideal $(\varphi_0,\varphi_1,\varphi_2)$ in ${
\kappa }[s,t] $. Hence, if $Syz(\varphi)$ is the sygyzy module of the
parameterization $\varphi$ for the curve $C$,
$(\alpha_0,\alpha_1,\alpha_2)\in Syz(\varphi)_k$.
\end{defn}

Now assume $d=2n$, and let us write explicitly the
parameterization of $C$: $\varphi _i=\varphi_{i0}s^{2n}+\dots+\varphi_{i,2n}t^{2n}$ for $i=0,1,2$.
 Consider a moving line in degree $n-1$ for $C$:
$\beta_0(s,t)x_0+\beta_1(s,t)x_1+\beta_2(s,t)x_2=0$ where
$\beta_i(s,t)=\sum _{k=0}^{n-1}B_{ik}s^kt^{n-1-k}$, satisfying the condition
$\beta_0(s,t)\varphi_0(s,t)+\beta_1(s,t)\varphi_1(s,t)+\beta_2(s,t)\varphi_2(s,t)\equiv 0,$ that is
\begin{equation}\label{ast1}
\sum _{i=0}^{2}\sum _{k=0}^{n-1}\sum _{j=0}^{2n}B_{ik}\varphi_{ij}s^{2n+k-j}t^{n-1-k+j}\equiv 0.
\end{equation}
Note that each monomial in $s$ and $t$ has total degree $3n-1$.
Rewriting \eqref{ast1} in terms of the powers $t^l$, we have
\begin{equation}\label{ast2}
\sum_{l=0}^{3n-1}\sum _{i=0}^{2}\sum _{a,b}B_{i,n-1-b}\varphi_{ia}s^{3n-1-l}t^l\equiv 0,
\end{equation}
where the inner sum is over all $a$ and $b$ such that $0\leq a\leq2n$, $0\leq b\leq n-1$ and
$a+b=l$.
This homogeneous polynomial is identically zero if and only if all of the coefficients are zero; i.e.,
if and only if for each 
$0\leq l\leq 3n-1$ 
we have
$$\sum _{i=0}^{2}\sum _{a,b}B_{i,n-1-b}\varphi_{ia}=0.$$

Hence to say that there exists a moving line in degree $n-1$ for $C$ is
equivalent to saying that the following linear system of $3n$
equations in the $3n$ variables $B_{00}, \dots, B_{2,n-1}$ has a non-trivial solution:
$$\underbrace{
\left(\begin{matrix}
\varphi_{0,2n}&\varphi_{1,2n}&\varphi_{2,2n}&0& \dots &0\\
&\dots& & &\dots &\\
0& \dots &0&\varphi_{0,0}&\varphi_{1,0}&\varphi_{2,0}
\end{matrix}\right)}_M
\left(\begin{matrix}
B_{00} \\
B_{10}\\
B_{20}\\
\vdots \\
B_{0,n-1}\\
B_{1,n-1}\\
B_{2,n-1}
\end{matrix}\right)
=
\left(\begin{matrix}
0 \\
0\\
0\\
\vdots \\
0\\
0\\
0
\end{matrix}\right)$$
where the $3n\times 3n$ matrix of the system $M(\varphi) = M = (m_{u,v})$ is defined
by $m_{u,v}$ as follows: writing $v=3w+i$ as a multiple of 3
with remainder $i$ (so $0\leq i \leq2$), then $m_{u,v}=\varphi_{i,2n+w-u}$ if
$0\leq 2n+w-u\leq 2n$, and $m_{u,v}=0$ otherwise.

So, if $d=2n$,  we have proved that there exists a
moving line in degree $n-1$ for $C$ if and only if $\det M=0$. \b
Moreover, ${\rm rk}\ M= 3n-p$ if and only if there are exactly $p$
independent moving lines  in degree $n-1$ for $C$.

In the odd degree case, $d=2n+1$, the same kind of computation
gives a condition analogous to \eqref{ast2}, with an equation of degree  $3n$ in
$s,t$; hence $M$ becomes a $(3n+1) \times 3n$ matrix, and the
analogous linear system has a non-trivial solution if and
only if ${\rm rk}\ M\leq 3n-1$; as before we find that there are
exactly $p$ independent moving lines  in degree $n-1$ for $C$ if and
only if ${\rm rk}\ M= 3n-p$.

Notice that if there are exactly $p$ independent moving lines, i.e. $\dim Syz(\varphi)_{n-1} =
p$, then there is a unique moving line of degree $n-p$, or,
equivalently, the splitting type of $C$ is $(n-p,n+p)$.

In summary, we have the following result (see also \cite[Proposition 5.3]{refSGD}):

\begin{thm}\label{splittingM}
Let $C\subset \pr2$ be a rational
curve of degree $d =2n+\delta$, $\delta \in \{0,1\}$, parameterized by
$\varphi _i=\varphi_{i0}s^{2n}+\dots+ \varphi_{i,2n}t^{2n}$ for $i=0,1,2$, and
define the $(3n+\delta)\times 3n$ matrix
$$M(\varphi)=M = (m_{u,v})_{0\leq u \leq 3n-1+\delta, 0\leq v \leq 3n-1}$$ as follows:
writing $v=3w+i$ as a multiple of 3
with remainder $i$ (so $0\leq i \leq2$), then $m_{u,v}=\varphi_{i,2n+w-u}$ if
$0\leq 2n+w-u\leq 2n$, and $m_{u,v}=0$ otherwise.
Then the splitting type of $C$ is $(n-p,n+p)$ if and only if ${\rm rk}\ M =
3n-p$.
\end{thm}

Thus Theorem \ref{splittingM} gives an algorithm to compute the
splitting type of every rational plane curve once we have a
parameterization for it.

\subsection{Finding parameterizations}\label{params}

Let $\pi:X\to\pr2$ be the blow up of $r\geq 3$ distinct points $p_i\in\pr2$.
Then, as noted earlier, the divisor classes of $L, E_1,\ldots,E_r$ give an integer basis
for ${\rm Cl}(X)$. This basis is, moreover, orthogonal with respect to the intersection form.
The intersection form is uniquely specified by $L^2=1$ and the fact that
$E_i^2=-1$ for all $i$.

The Weyl group $W(X)$ is a subgroup of orthogonal transformations on ${\rm Cl}(X)$.
It is generated by the elements $s_0,\ldots,s_{r-1}\in W(X)$ where
$s_i(D)=D+(v_i\cdot D)v_i$ and where $v_0=[L-E_1-E_2-E_3]$ and
$v_i=[E_i-E_{i+1}]$ for $0<i<r$.
Note that $W(X)$ really depends only on $r$, not on $X$.
The action of the elements $s_1,\ldots,s_{r-1}$ corresponds to permuting the classes of the divisors
$E_i$, and the action of $s_0$ corresponds to a quadratic Cremona transformation $T$
centered at $p_1, p_2$ and $p_3$. But whereas $T$ is defined only when the points
$p_1, p_2$ and $p_3$ are not collinear, $s_0$ is defined formally and thus it
and every other element of $W(X)$ makes sense
regardless of the positions of the points $p_i$. It is useful to note that
$s_v(D)=D+(v\cdot D)v\in W(X)$ whenever $v=v_{ij}$ for $v_{ij}=[E_i-E_j]$, or
$v=v_{ijk}$ for $v_{ijk}=[L-E_i-E_j-E_k]$ with $i<j<k$. When $v=v_{ijk}$,
$s_v$ corresponds to a quadratic Cremona transformation $T'$ centered at $p_i,p_j,p_k$.
Such a $T'$ can be given explicitly. For example, let $H_{ij},H_{ik},H_{jk}\in R=\KK[x_0,x_1,x_3]$
be linear forms where $H_{ij}$ defines the line through $p_i$ and $p_j$, and likewise
for $H_{ik}$ and $H_{jk}$. Then a specific such $T'$ can be given by defining
$T'(p)$ for any point $p=[a,b,c]$ other than $p_i,p_j$ and $p_k$ by
$T'(p)=[(H_{ij}H_{ik})(a,b,c), (H_{ij}H_{jk})(a,b,c), (H_{ik}H_{jk})(a,b,c)]$,
where, for example, $(H_{ij}H_{ik})(a,b,c)$ means evaluation of the
form $H_{ij}H_{ik}$ at $(a,b,c)$. Moreover, the linear system of sections
of $D=dL-m_1E_1-\cdots-m_rE_r$ can be identified with those of $s_{v_{ijk}}D$,
but where we regard $D$ as being with respect to blowing up the points
$p_1,\cdots,p_r$, under the identification we must regard $s_{v_{ijk}}D$
as being with respect to blowing up the points $q_1,\ldots,q_r$, where $q_l=T'(p_l)$
for $l\not\in\{i,j,k\}$, and where $q_i=(1,0,0)$, $q_j=(0,1,0)$ and $q_k=(0,0,1)$.

If $C=\pi(C')$ is a curve of positive degree, where $C'$ is some
an exceptional curve on $X$, then as long as $r\geq 3$ and the points $p_1,\ldots,p_r$
are sufficiently general, there is in fact \cite{refN} a Cremona transformation $T:\pr2 -\,\!-\!\to\pr2$
whose locus of indeterminacy is contained in the set of points $p_i$
and such that the image of $C'$ is a line. One can find such a $T$
by finding a sequence of elements $u_1=v_{i_1j_1k_1},\ldots, u_l=v_{i_lj_lk_l}$
such that $w=s_{u_1}\cdots s_{u_l}$ where $w\in W(X)$ is an element
such that $w([C'])$ has the form $[L-E_i-E_j]$ for some $i<j$.
Composing the quadratic Cremona transformations corresponding to each $s_{u_i}$
gives the desired transformation $T$. Pulling the parameterization of the line
corresponding to $L-E_i-E_j$ back via $T$ gives a parameterization of $\pi(C')$.

\begin{example}
Suppose we would like to parameterize a quartic $Q$ having
double points at $p_1,p_2$ and $p_3$
and passing through $p_4,\ldots,p_8$ for appropriately general points $p_i\in\pr2$.
The class of the proper transform of $Q$ is $F=[4L-2E_1-2E_2-2E_3-E_4-\cdots-E_8]$.
We choose $u_1=[L-E_i-E_j-E_k]$ so that $u_1\cdot F$ is as small as possible, hence
$u_1=[L-E_1-E_2-E_3]$ and we define $F_1=s_{u_1}(F)=[2L-E_4-\cdots-E_8]$.
There are several choices for $u_2=[L-E_i-E_j-E_k]$ such that $u_2\cdot F_1$
is as small as possible; e.g., $u_2=[L-E_6-E_7-E_8]$,
which gives $F_2=s_{u_2}(F_1)=[L-E_4-E_5]$.
Let $T_1$ be a Cremona transformation corresponding to $s_{u_1}$
(i.e., $T_1$ is a quadratic transformation centered at $p_1,p_2$ and $p_3$)
and let $T_2$ be a Cremona transformation corresponding to
$s_{u_2}$ (i.e., $T_2$ is a quadratic transformation centered at $T_1(p_6),T_1(p_7),T_1(p_8)$).
If we take $T=T_2T_1$, then $T(Q)$ is the line $\Lambda$ through
$q_4=T_2T_1(p_4)$ and $q_5=T_2T_1(p_5)$. If $\varphi$ is a
parameterization of this line (given for example by the pencil of lines through any
point not on the line $\Lambda$), then $T^{-1}\circ\varphi$ parameterizes $Q$.

In general it is easier to compute the successive parameterizations
$T_l^{-1}\circ\varphi$, $T_{l-1}^{-1}\circ T_l^{-1}\circ\varphi$, $\ldots$,
$T^{-1}\circ\varphi=T_1^{-1}\circ\cdots\circ T_l^{-1}\circ\varphi$,
since this involves working in the ring $S=\KK[s,t]$ which has only two variables, rather than
first finding $T^{-1}$ (which would be in terms of elements of
$R=\KK[x_0,x_1,x_2]$) and then composing with $\varphi$.
Also, one should keep in mind that
at each successive parameterization, the component functions may have a common factor
which must be removed. For example, for our parameterization of $Q$ we can take
the pencil of curves composing the
complete linear system $|s_{u_1}s_{u_2}(L-E_1)|$. But
$s_{u_1}s_{u_2}(L-E_1)=3L-2E_1-E_2-E_3-E_6-E_7-E_8$,
so the components of the parameterization of $Q$ are the restrictions to $Q$
of a basis of the cubics singular at $p_1$ and passing through $p_2,p_3,p_6,p_7$ and $p_8$.
But at each stage, $T_i$ is quadratic, so the composition $T^{-1}$ is defined in terms of
homogeneous polynomials of degree $2l$, which in our case would be 4.
The difference is accounted for by the components of $T^{-1}|_\Lambda$ having a common factor, in this case
a common linear factor.
\end{example}

\subsection{Looking for sygyzies coming from the plane}\label{dal piano}
 Here we explain the idea which underlies Conjecture \ref{rptconj}. We
continue with the notation established earlier in this section.
Let $\varphi=(\varphi_0,\varphi_1,\varphi_2)$ be a parameterization
of a rational plane curve of degree $d$, let $J=(\varphi_0,\varphi_1,\varphi_2)\subset S$
be the ideal generated by the components of $\varphi$, and let
$\widetilde \varphi$ be the associated ring map given in \eqref{eqn2}.
Consider a syzygy $(\alpha_0, \alpha_1, \alpha_2)$ of $J$; thus
$\alpha_0\varphi_0+ \alpha_1\varphi_1 + \alpha_2\varphi_2 =0$
for some $\alpha_i\in S_k$ for some $k$.

\begin{defn}\label{from the plane}
We say that the sygyzy $(\alpha_0, \alpha_1, \alpha_2)$
\emph{comes from $\pr 2$} if there are $A_0,A_1,A_2
\in K[\pr 2]_q $, such that $A_0x_0+A_1x_1+A_2x_2=0$, and there exists $p
\in S_{dq-k}$, with $\widetilde \varphi(A_i) = p\alpha_i$.
Hence a sygyzy $(\alpha_0, \alpha_1, \alpha_2)$ comes from $\pr 2$ if and only if
there exists $q\in {\mathbb N}$ such that the following colon ideal is nonzero in degree $dq-k$; i.e.,
$$ [(\varphi_0,\varphi_1,\varphi_2)^q:(\alpha_0,\alpha_1,\alpha_2)]_{dq-k} \neq
\{0\}.$$
\end{defn}

\begin{example}\label{8,3,3} Given the blow up $X$ of $\pr2$ at general points $p_1,\ldots,p_r$,
the first and simplest example of a non-Ascenzi rational curve $C$ with
unbalanced splitting type has class $[C]=[8L-3E_1-3E_2-\dots -3E_7]$.
Indeed, the splitting type of $C$ is $(3,5)$ (see \cite{refFHH} and
\cite{refGHI2}). Notice that $C$ is not exceptional, but
$[8L-3E_1-3E_2-\dots -3E_7-E_8-E_9]$ is the class of an exceptional curve,
and moreover this exceptional curve has the same splitting type as does $C$,
and our discussion applies equally well to this exceptional curve as it does to $C$.

Given that the splitting type is $(3,5)$, Lemma \ref{b1} tells us
that the image $\overline C$ of $C$ in $\pr2$ has a
parameterization $\varphi$ with a syzygy of degree 3. This syzygy
in fact comes from $\pr2$. To see this, note that the linear
system of cubics through 7 general points has a linear syzygy. In
fact, there is a relation $A_0x_0+A_1x_1+A_2x_2=0$, where
$A_0,A_1, A_2$ is a basis for $|3L-E_1-\cdots-E_7|$ regarded as a
subspace of the space of degree 3 forms $|3L|$ on $\pr2$. Moreover
$(3L-E_1-\cdots-E_7)\cdot C=3$, hence the relation on the $A_i$'s
forces a relation of degree $k=3$ among the restrictions
$\varphi_i$ of the $L_i$'s to $C$. Since $\overline C$ has degree
$d=8$, and each $A_i$ has degree $q=3$, we see each
$\widetilde\varphi(A_i)$ has degree $24=dq$. The common factor $p$
is a form on $\pr1$ of degree 21 vanishing on the inverse images
in $\pr1$ of the 7 points $p_i\in\overline C$, and we have
$\deg(p)=21=dq-k$. Even if we did not already know that $a_C=3$,
we would see by Lemma \ref{b1} that $a_C\leq 3$. Since $a_C\geq3$
by Lemma \ref{splitlem}, we recover the fact that $a_C=3$.
\end{example}

If, as is the case in the preceding example,
$a_C$ is smaller than ``expected'' (i.e., smaller than $\lfloor\frac{d}{2}\rfloor$), then the
three forms $\varphi_0,\varphi_1,\varphi_2$ do not give a generic
$g^2_d$, since for three generic forms of degree $d$ over $\pr 1$ the
minimal degree of a syzygy is $\lfloor\frac{d}{2}\rfloor$ by \cite{refAs1}. All known
examples lead us to conjecture that this ``non-generic situation"
occurs only when the presence of a syzygy of smaller than expected degree is
``forced" by something happening in $\pr 2$, i.e. because the syzygy comes from $\pr2$,
and thus there is a linear relation
in $k[\pr 2]$ among polynomials whose restriction to $C$ has degree
$a<\lfloor\frac{d}{2}\rfloor$ (modulo taking away common factors).
Thus Conjecture \ref{rptconj} comes from the expectation that what happened
in Example \ref{8,3,3} should be what always happens when we get unbalanced splitting.
The explicit statement we give for Conjecture \ref{rptconj} includes some additional
restrictions on the $A_i$ (that they be sections of a divisor $A$ of a particular form),
since examples lead us to expect these additional restrictions are always satisfied.

\begin{remark}\label{NBAC} 
Suppose $C$ is an Ascenzi plane curve with non-balanced splitting.
Thus $C$ is of the form $(d,m_1,\dots ,m_r)$ where we may assume $m_1>m_2\geq\dots \geq m_r$,
and we have $d<2m_1-1$. Then Conjecture 2 holds with $A=L-E_1$:
there is a syzygy of degree $a_C=d-m_1$ on the parameterization functions
$\varphi_i$ parameterizing $C$; this syzygy comes from $\pr2$, specifically
from the linear syzygy on the pencil of lines through $p_1$, and we have
$a_C=(dL-m_1E_1-\cdots-m_rE_r)\cdot (L-E_1)$ as asserted by Conjecture \ref{rptconj}.
See also Remark \ref{AEpairsrem}.
\end{remark}

\begin{remark}\label{K_7} In Example \ref{8,3,3} we have $C^2=1$
and $A = -K_7$, where $-K_7=3L-E_1-\cdots-E_7$ is an anticanonical divisor for
the blow up of $\pr2$ at 7 points. In fact,
$h^0(C)=3$ and $|C|$ is a homaloidal net (i.e., $[C]$ is in the Weyl group orbit of $[L]$),
so a smooth irreducible curve $C_2\in |2C|$
is rational, and we have $C_2\cdot A=2a_C$. Thus by Conjecture \ref{rptconj}
we expect $a_{C_2}\leq 2a_C$; i.e., that the splitting gap will be (at least) double for $C_2$, and
this is indeed the case, as we now show. More generally,
consider a smooth irreducible curve  $C_r\in|rC-(r-1)E_8|$. The fact that
the splitting types of $C_r$ are $(3r,5r)$ follows by
applying Lemma \ref{splitlem} together with the forthcoming Proposition \ref{unbalsplitting},
using $A = 3L- E_1- \dots - E_7$ as in Example \ref{AEpairs}.
In the same way we can construct a plethora of similar examples, such as
a curve $C'_4$ of type $(32,12,12,12,12,12,12,12,2,2,2)$ and
a curve $C'_6$ of type $(48,18,18,18,18,18,18,18,18,3,2,2,2,2,2,2,2)$,
whose splitting types are $4(3,5)=(12,20)$ and $6(3,5)=(18,30)$, respectively.
\end{remark}

\section{Splitting type of rational plane curves with specified singularities}\label{sect3}

In this section we will consider the splitting type for rational curves of the form $\pi(D)$,
where $D\subset X$ is a smooth rational curve and $\pi:X\to\pr2$ is
the blowing up of $\pr2$ at $r$ distinct points $p_1,\ldots,p_r$.
Note that if we write $[D]=[d_DL-m_1E_1-\cdots-m_rE_r]$,
there is no loss of generality in assuming that $m_D=m_1$ is the maximum of the $m_i$.

\subsection{New bounds on splitting types}\label{newbnds}
By Lemma \ref{splitlem}, we have the bound $a_D\leq d_D-m_D=A\cdot D$, where
$A=L-E_1$. As Remark \ref{AEpairsrem} will explain, the following proposition
generalizes this bound. It shows that we can sometimes get better
bounds on $a_D$ by finding other divisors $A$ such that
$\mu_A$ has nontrivial kernel.

\begin{prop}\label{unbalsplitting}
Let $X$ be obtained by blowing up $r$ distinct points
of $\pr2$. Let $L, E_1,\ldots,E_r$ be the corresponding basis of the
divisor class group of $X$. Let $D\subset X$ be a smooth rational curve
and let $A$ be a divisor such that $h^1(A)=0$, $H^0(A-D+L)=0$
and ${\rm le}(A)\ge1$. Then $a_D\leq A\cdot D$, and equality
holds if, moreover, $D$ is exceptional such that ${\rm le}(A)={\rm le}(A+D)$ and $\mu_A$ is surjective.
\end{prop}

\begin{proof} Since ${\rm le}(A)\ge1$, we must have $h^0(A)>1$,
hence $A\cdot L>0$, so $h^2(A)=0$. Because $h^1(A)=0$ by hypothesis,
by taking cohomology of
$$0 \to \OO_X(A) \to \OO_X(A+L) \to \OO_L(A+L) \to 0 $$
we see that $h^1(A+L)=0$. From the diagram
$$\begin{matrix}
0 & \to & \OO_X(A-D)\otimes H^0(L)  & \to & \OO_X(A)\otimes H^0(L)
& \to & \OO_D(A)\otimes H^0(L) & \to & 0 \cr
{} & {} & \downarrow  & {} & \downarrow & {} & \downarrow & {} & {} \cr
0 & \to & \OO_X(A-D+L)  & \to & \OO_X(A+L) & \to & \OO_D(A+L) & \to & 0 \cr
\end{matrix}$$
we get the following diagram, which has exact rows, by taking cohomology:
$$\begin{matrix}
0 & \to & H^0(A)\otimes H^0(L)  & \to & H^0(\OO_D(A))\otimes H^0(L)
& \to & H^1(A-D)\otimes H^0(L) & \to & 0 \cr
{} & {} & \downarrow \raise3pt\hbox to0in{$\scriptstyle\mu_A$\hss} & {}
& \downarrow\raise3pt\hbox to0in{$\scriptstyle\mu_2$\hss} & {}
& \downarrow\raise3pt\hbox to0in{$\scriptstyle\mu _3$\hss} &
{} & {} \cr
0 & \to & H^0(A+L) & \to & H^0(\OO_D(A+L))
& \to & H^1(A-D+L)& \to & 0 \cr
\end{matrix}$$
Thus we get an inclusion ${\rm ker}(\mu_A)\subseteq {\rm ker}(\mu_2)$ but by
\eqref{eqnstar} we have ${\rm ker}(\mu_2)=
H^0(\OO_D(A\cdot D-a_D))\oplus H^0(\OO_D(A\cdot D-b_D))$.
Since $a_D\leq b_D$, this means $0< h^0(\OO_D(A\cdot D-a_D))$,
hence $a_D\leq A\cdot D$. For the rest, a similar argument applied to
$$\begin{matrix}
0 & \to & \OO_X(A)\otimes H^0(L)  & \to & \OO_X(A+D)\otimes H^0(L)
& \to & \OO_D(A+D)\otimes H^0(L) & \to & 0 \cr
{} & {} & \downarrow  & {} & \downarrow & {} & \downarrow & {} & {} \cr
0 & \to & \OO_X(A+L)  & \to & \OO_X(A+D+L) & \to & \OO_D(A+D+L) & \to & 0 \cr
\end{matrix}$$
shows that
$H^0(\OO_D(A\cdot D-a_D-1))\oplus H^0(\OO_D(A\cdot D-b_D-1))=0$
if ${\rm le}(A)={\rm le}(A+D)$ and $\mu_A$ is surjective, and hence
that $A\cdot D<a_D+1$, which gives $A\cdot D=a_D$.
\end{proof}

The following lemma will be useful in identifying candidate divisors $A$.

\begin{lem}\label{lelemma} Let $X$ be obtained by blowing up $r$ distinct points
of $\pr2$. Let $A=dL-\sum_im_iE_i$ be a divisor on
$X$ such that $d\ge0$, $-K_X\cdot A=2$ and $h^1(A)=0$.
\begin{itemize}
\item[(a)] If $3h^0(A)-h^0(A+L)\ge1$
(in which case we know ${\rm le}(A)\ge1$),
then $A^2+1\geq L\cdot A$.
\item[(b)] If $A^2+1\geq L\cdot A$, then ${\rm le}(A)\ge1$.
\end{itemize}
\end{lem}

\begin{proof}
>From $h^1(A)=0$, $A\cdot L=d \ge 0$ and
$$0 \to \OO_X(A) \to \OO_X(A+L) \to \OO_L(A+L) \to 0 $$
we see that $h^0(A+L)=h^0(A)+A\cdot L+2$. Since $d\ge0$, we have $h^2(A)=0$.
Since also $h^1(A)=0$ and $-A\cdot K_X=2$,
we have $2h^0(A)=A^2-A\cdot K+2=A^2+4$ by Riemann-Roch,
which says $h^0(A)-h^1(A)+h^2(A)=(A^2-K_X\cdot A)/2+1$.
(a) This is just a calculation:
$1\le3h^0(A)-h^0(A+L)=2h^0(A)-A\cdot L-2
=A^2+2-A\cdot L$, which gives the result.

(b) This time we have ${\rm le}(A)\geq 3h^0(A)-h^0(A+L)=2h^0(A)-A\cdot L-2
=A^2+2-A\cdot L\ge1$.
\end{proof}

\begin{remark}\label{AEpairsrem}
Assume $X$ is obtained by blowing up $r\ge1$ distinct points $p_i$ of $\pr2$,
and consider a smooth rational curve $D\subset X$ with $[D]=[d_DL-m_1E_1-\cdots-m_rE_r]$.
Assume that $m_D=m_1$
is the maximum of the $m_i$. Now let $A=L-E_1$. By Lemma \ref{lelemma}(b), we have ${\rm le}(A)\geq1$, although it is of course obvious in this case that
$A$ has a nontrivial linear syzygy and it is not hard to check that in fact
${\rm le}(A)=1$. If $d_D>2$
we have $(A-D+L)\cdot L<0$ and hence $h^0(A-D+L)=0$ so by Proposition \ref{unbalsplitting}
we obtain $a_D\leq (L-E_1)\cdot D=d_D-m_D$, which
is just the upper bound given in Lemma \ref{splitlem}.
\end{remark}

\begin{example}\label{AEpairs}
Here we assume $X$ is obtained by blowing up $r=9$ generic points $p_i$ of $\pr2$.
(We pick $r=9$ to have a single value of $r$ big enough to accommodate
the discussion in this example.)
Two non-Ascenzi types of plane rational curves with generically situated singularities
were previously known to have unbalanced splitting (both with gap 2), namely
$(8,3,3,3,3,3,3,3,0,0)$ \cite[Lemma 3.12(b)(ii)]{refFHH} and
$(12, 5, 5, 5, 4, 4, 4, 4, 2,0)$ \cite[\S A2.1]{refGHI1}.
We show how our results recover the splittings in these cases.
We focus on the first case (the
second case can be done exactly the same way).
First consider $A=3L-E_1-\cdots-E_7$.
Since the points $p_i$ are not special, it's clear that
$h^1(A)=0$, so Lemma \ref{lelemma}(b) implies that ${\rm le}(A)\geq1$, but
by \cite[Theorem IV.1]{refH2} we know $\mu_A$ is surjective. Since
$h^0(A)=3$ and $h^0(A+L)=8$, we see in fact ${\rm le}(A)=1$.
Since $[8L-3(E_1+\cdots+E_7)]$ is in the $W(X)$-orbit of $[L]$,
we know (see \cite{refN}) that there is a
smooth rational curve whose class is $[8L-3(E_1+\cdots+E_7)]$;
let $C$ be any such curve.
We have $h^0(A-C+L)=0$ since $(A-C+L)\cdot L<0$.
Thus $a_C\leq A\cdot C=3$ by Proposition \ref{unbalsplitting}.
>From Ascenzi's lower bound $a_C\geq m_C=3$ we see that we actually have $a_C=3$ here.
(If we instead consider the exceptional curve $E$ whose class is
$[8L-3(E_1+\cdots+E_7)-E_8-E_9]$, then the same argument shows
that $a_E=3$, but moreover it is also true that ${\rm le}(A+E)=1$,
and thus $a_E=3$ would follow from
Proposition \ref{unbalsplitting} alone, but the simplest argument to show ${\rm le}(A+E)=1$
involves using the fact that $a_E=3$.)
\end{example}

\begin{example}\label{MoreAEpairs}
Again let $X$ be the blow up of $\pr2$ at $r$ generic points.
Here we determine the splitting for several non-Ascenzi exceptional curves $E$ using
the same method as in the previous example but with different choices for $A$.
First assume $[E]=[12L -5(E_1+\cdots+E_4)-3(E_5+\cdots+E_9)]$
but this time take $A=5L-2(E_1+\cdots+E_4)-(E_5+\cdots+E_9)$.
It's clear that $A$ is effective and nef (since $A=D-K_X$ for $D=2L-E_1-\cdots-E_4$ and both
$D$ and $-K_X$ are effective and nef), and since $-K_X\cdot A=2$ it follows from
\cite{refH} that $h^1(A)=0$. Thus Lemma \ref{lelemma}(b) implies that ${\rm le}(A)\geq1$.
As before, we have $(A-E+L)\cdot L<0$ so $h^0(A-E+L)=0$.
Thus $a_E\leq A\cdot E=5$ by Proposition \ref{unbalsplitting} and using Ascenzi's
lower bound we again have equality, $a_E=5$.

Here are a few additional pairs which work the same way. For simplicity we give only the numerical types
corresponding to $A$ and $E$. In each case we obtain $a_E=m_E$:
$$\begin{array}{llll}
A:& (7, 3, 3, 3, 2, 2, 2, 1, 1, 1, 1) &E:& (18, 8, 8, 8, 6, 6, 5, 3, 3, 3, 3)\\
A:& (9, 4, 3, 3, 3, 3, 3, 2, 2, 2)      &E:& (20, 9, 7, 7, 7, 7, 7, 5, 5, 5)\\
A:& (7, 3, 3, 2, 2, 2, 2, 2, 2, 1)      &E:& (20, 9, 9, 7, 6, 6, 6, 6, 6, 3, 1)\\
A:& (7, 3, 3, 3, 3, 1, 1, 1, 1, 1, 1, 1) &E:& (20, 9, 9, 9, 9, 4, 4, 3, 3, 3, 3, 3)
\end{array}$$
It is not always so easy to determine $a_E$ exactly.
For example, if $E$ is an exceptional curve of type $(40, 15, 15, 15, 13, 13, 13, 13, 13, 9)$,
then $A=(E+K_X+L)/2$ has type $(19, 7, 7, 7, 6, 6, 6, 6, 6, 4)$,
and $h^1(A)=0$ by the methods of \cite{refH}, while
$h^0(A-E+L)=0$ since $(A-E+L)\cdot L<0$.
Applying Lemma \ref{lelemma}(b)
we have ${\rm le}(A)\geq1$, but Proposition \ref{unbalsplitting}
and Ascenzi's bounds give only $15\leq a_E\leq A\cdot E=19$. Computer
calculations indicate that in fact $a_E=19$, as predicted by Conjecture \ref{9ptconj}.
\end{example}

Each $A$ in Example \ref{MoreAEpairs} has $-K_X\cdot A=2$. For reasons that so far remain mysterious,
when an exceptional curve $E$ has unbalanced splitting
it always seems possible to find an $A$ such that not only do we have ${\rm le}(A)=1$
and $E\cdot A=a_E$, but in addition such that $-K_X\cdot A=2$.

\subsection{Smooth rational curves on 7 point blow ups}\label{7pts}

Here we classify all classes $[C]=[dL-m_1E_1-\cdots-m_rE_r]$ where
$C$ is a smooth rational curve on the blow up $X$ of $\pr2$ at $r\leq7$ generic points.
Since the case $r=7$ subsumes $r<7$, we will assume $r=7$. As we will see,
$r=7$ is the least $r$ such that there are infinitely many non-Ascenzi $C$; moreover,
all but finitely many of these are unbalanced. The method we use here can be used to
find all Ascenzi and all non-Ascenzi $C$ when $r=8$, but there will be many more cases
to analyze if one wants also to determine the splitting types.
For $r\leq8$, the Weyl group $W(X)$ is finite. The case $r>8$ will be more difficult,
at least partly due to the fact that $W(X)$ is then infinite.

In the next result, we show that the class of every smooth rational curve $C$
on the blow up $X$ of $\pr2$ at 7 generic points is in the Weyl group orbit either
of $E_7$, $H_0+kH_1$,  $H_2+kH_1$, $2H_0$ or of $H_1$,
where $H_0=L$, $H_1=L-E_1$ and $H_2=2L-E_1-E_2$, and $k\in \bf N$.

\begin{thm}\label{list7} Let $X$ be the blow up of $\pr2$ at $r\leq7$ generic points. The numerical types
$(d,m_1,\ldots,m_7)$ of all smooth rational $C\subset X$, up to permutations of the $m_i$'s, are given in the following lists, where the corresponding splitting gap $\gamma _C$ in each case which is not Ascenzi is given.

The types for the orbit of $E_7$ are
$(0$, $0$, $0$, $0$, $0$, $0$, $0$, $-1)$,
$(1$, $1$, $1$, $0$, $0$, $0$, $0$, $0)$,
$(2$, $1$, $1$, $1$, $1$, $1$, $0$, $0)$ and
$(3$, $2$, $1$, $1$, $1$, $1$, $1$, $1)$, all of which are Ascenzi.

The types for the orbit of $H_0+kH_1$ are (here $k\in \bf N$):\newline
\hbox to.5in{\hfil} $(1, 0, 0, 0, 0, 0, 0, 0)+k(1, 1, 0, 0, 0, 0, 0, 0)$, which is Ascenzi\newline
\hbox to.5in{\hfil} $(2, 1, 1, 1, 0, 0, 0, 0)+k(1, 1, 0, 0, 0, 0, 0, 0)$, which is Ascenzi\newline
\hbox to.5in{\hfil} $(2, 1, 1, 1, 0, 0, 0, 0)+k(2, 1, 1, 1, 1, 0, 0, 0)$, which is Ascenzi\newline
\hbox to.5in{\hfil} $(3, 2, 1, 1, 1, 1, 0, 0)+k(1, 1, 0, 0, 0, 0, 0, 0)$, which is Ascenzi\newline
\hbox to.5in{\hfil} $(3, 2, 1, 1, 1, 1, 0, 0)+k(2, 1, 1, 1, 1, 0, 0, 0)$, which is Ascenzi\newline
\hbox to.5in{\hfil} $(3, 2, 1, 1, 1, 1, 0, 0)+k(3, 2, 1, 1, 1, 1, 1, 0)$, which is Ascenzi\newline
\hbox to.5in{\hfil} $(4, 2, 2, 2, 1, 1, 1, 0)+k(2, 1, 1, 1, 1, 0, 0, 0)$, which is Ascenzi\newline
\hbox to.5in{\hfil} $(4, 2, 2, 2, 1, 1, 1, 0)+k(3, 2, 1, 1, 1, 1, 1, 0)$, which is Ascenzi\newline
\hbox to.5in{\hfil} $(4, 2, 2, 2, 1, 1, 1, 0)+k(4, 2, 2, 2, 1, 1, 1, 1)$, which is Ascenzi\newline
\hbox to.5in{\hfil} $(4, 3, 1, 1, 1, 1, 1, 1)+k(3, 2, 1, 1, 1, 1, 1, 0)$, which is Ascenzi\newline
\hbox to.5in{\hfil} $(5, 2, 2, 2, 2, 2, 2, 0)+k(3, 2, 1, 1, 1, 1, 1, 0)$, which is Ascenzi\newline
\hbox to.5in{\hfil} $(5, 2, 2, 2, 2, 2, 2, 0)+k(5, 2, 2, 2, 2, 2, 2, 1)$, which is Ascenzi if and only if $k = 0$;
$\gamma_C=|k-1|$\newline
\hbox to.5in{\hfil} $(5, 3, 2, 2, 2, 1, 1, 1)+k(2, 1, 1, 1, 1, 0, 0, 0)$, which is Ascenzi\newline
\hbox to.5in{\hfil} $(6, 3, 3, 2, 2, 2, 2, 1)+k(3, 2, 1, 1, 1, 1, 1, 0)$, which is Ascenzi\newline
\hbox to.5in{\hfil} $(6, 3, 3, 2, 2, 2, 2, 1)+k(4, 2, 2, 2, 1, 1, 1, 1)$, which is Ascenzi\newline
\hbox to.5in{\hfil} $(6, 3, 3, 2, 2, 2, 2, 1)+k(5, 2, 2, 2, 2, 2, 2, 1)$, which is Ascenzi if and only if $k < 2$;
$\gamma_C=k$\newline
\hbox to.5in{\hfil} $(7, 3, 3, 3, 3, 2, 2, 2)+k(4, 2, 2, 2, 1, 1, 1, 1)$, which is Ascenzi\newline
\hbox to.5in{\hfil} $(7, 3, 3, 3, 3, 2, 2, 2)+k(5, 2, 2, 2, 2, 2, 2, 1)$, which is Ascenzi if and only if $k = 0$;
$\gamma_C=k+1$\newline
\hbox to.5in{\hfil} $(8, 3, 3, 3, 3, 3, 3, 3)+k(5, 2, 2, 2, 2, 2, 2, 1)$, which is never Ascenzi;
$\gamma_C=k+2$\newline

The types for the orbit of $H_2+kH_1$ are:\newline
\hbox to.5in{\hfil} $( 2, 1, 1, 0, 0, 0, 0, 0)+k(1, 1, 0, 0, 0, 0, 0, 0)$, which is Ascenzi\newline
\hbox to.5in{\hfil} $( 3, 2, 1, 1, 1, 0, 0, 0)+k(1, 1, 0, 0, 0, 0, 0, 0)$, which is Ascenzi\newline
\hbox to.5in{\hfil} $( 3, 2, 1, 1, 1, 0, 0, 0)+k(2, 1, 1, 1, 1, 0, 0, 0)$, which is Ascenzi\newline
\hbox to.5in{\hfil} $( 4, 2, 2, 2, 1, 1, 0, 0)+k(2, 1, 1, 1, 1, 0, 0, 0)$, which is Ascenzi\newline
\hbox to.5in{\hfil} $( 4, 3, 1, 1, 1, 1, 1, 0)+k(1, 1, 0, 0, 0, 0, 0, 0)$, which is Ascenzi\newline
\hbox to.5in{\hfil} $( 4, 3, 1, 1, 1, 1, 1, 0)+k(3, 2, 1, 1, 1, 1, 1, 0)$, which is Ascenzi\newline
\hbox to.5in{\hfil} $( 5, 3, 2, 2, 2, 1, 1, 0)+k(2, 1, 1, 1, 1, 0, 0, 0)$, which is Ascenzi\newline
\hbox to.5in{\hfil} $( 5, 3, 2, 2, 2, 1, 1, 0)+k(3, 2, 1, 1, 1, 1, 1, 0)$, which is Ascenzi\newline
\hbox to.5in{\hfil} $( 6, 3, 3, 3, 2, 1, 1, 1)+k(2, 1, 1, 1, 1, 0, 0, 0)$, which is Ascenzi\newline
\hbox to.5in{\hfil} $( 6, 3, 3, 3, 2, 1, 1, 1)+k(4, 2, 2, 2, 1, 1, 1, 1)$, which is Ascenzi\newline
\hbox to.5in{\hfil} $( 6, 4, 2, 2, 2, 2, 1, 1)+k(3, 2, 1, 1, 1, 1, 1, 0)$, which is Ascenzi\newline
\hbox to.5in{\hfil} $( 6, 3, 3, 2, 2, 2, 2, 0)+k(3, 2, 1, 1, 1, 1, 1, 0)$, which is Ascenzi\newline
\hbox to.5in{\hfil} $( 7, 4, 3, 3, 2, 2, 2, 1)+k(3, 2, 1, 1, 1, 1, 1, 0)$, which is Ascenzi\newline
\hbox to.5in{\hfil} $( 7, 4, 3, 3, 2, 2, 2, 1)+k(4, 2, 2, 2, 1, 1, 1, 1)$, which is Ascenzi\newline
\hbox to.5in{\hfil} $( 8, 4, 4, 3, 3, 2, 2, 2)+k(4, 2, 2, 2, 1, 1, 1, 1)$, which is Ascenzi\newline
\hbox to.5in{\hfil} $( 8, 4, 3, 3, 3, 3, 3, 1)+k(5, 2, 2, 2, 2, 2, 2, 1)$, which is Ascenzi if and only if $k < 2$; $\gamma_C=k$\newline
\hbox to.5in{\hfil} $( 9, 4, 4, 4, 3, 3, 3, 2)+k(4, 2, 2, 2, 1, 1, 1, 1)$, which is Ascenzi\newline
\hbox to.5in{\hfil} $( 9, 4, 4, 4, 3, 3, 3, 2)+k(5, 2, 2, 2, 2, 2, 2, 1)$, which is Ascenzi if and only if $k = 0$; $\gamma_C=k+1$\newline
\hbox to.5in{\hfil} $(10, 4, 4, 4, 4, 4, 3, 3)+k(5, 2, 2, 2, 2, 2, 2, 1)$, which is never Ascenzi; $\gamma_C=k+2$\newline

The types for the orbit of $2H_0$ are:\newline
\hbox to.5in{\hfil} $( 2, 0, 0, 0, 0, 0, 0, 0)$, which is Ascenzi \newline
\hbox to.5in{\hfil} $( 4, 2, 2, 2, 0, 0, 0, 0)$, which is Ascenzi \newline
\hbox to.5in{\hfil} $( 6, 4, 2, 2, 2, 2, 0, 0)$, which is Ascenzi \newline
\hbox to.5in{\hfil} $( 8, 4, 4, 4, 2, 2, 2, 0)$, which is Ascenzi \newline
\hbox to.5in{\hfil} $( 8, 6, 2, 2, 2, 2, 2, 2)$, which is Ascenzi \newline
\hbox to.5in{\hfil} $(10, 6, 4, 4, 4, 2, 2, 2)$, which is Ascenzi \newline
\hbox to.5in{\hfil} $(10, 4, 4, 4, 4, 4, 4, 0)$, which is not Ascenzi; $\gamma_C=0$ \newline
\hbox to.5in{\hfil} $(12, 6, 6, 4, 4, 4, 4, 2)$, which is Ascenzi \newline
\hbox to.5in{\hfil} $(14, 6, 6, 6, 6, 4, 4, 4)$, which is not Ascenzi; $\gamma_C=2$ \newline
\hbox to.5in{\hfil} $(16, 6, 6, 6, 6, 6, 6, 6)$, which is not Ascenzi; $\gamma_C=4$ \newline

The types for the orbit of $H_1$ are: \newline
\hbox to.5in{\hfil} $(1, 1, 0, 0, 0, 0, 0, 0)$, which is Ascenzi \newline
\hbox to.5in{\hfil} $(2, 1, 1, 1, 1, 0, 0, 0)$, which is Ascenzi \newline
\hbox to.5in{\hfil} $(3, 2, 1, 1, 1, 1, 1, 0)$, which is Ascenzi \newline
\hbox to.5in{\hfil} $(4, 2, 2, 2, 1, 1, 1, 1)$, which is Ascenzi \newline
\hbox to.5in{\hfil} $(5, 2, 2, 2, 2, 2, 2, 1)$, which is Ascenzi \newline
\end{thm}

\begin{proof} Let $C$ be a smooth rational curve on $X$.
Because $r=7$, $X$ is a Del Pezzo surface, and hence $-K_X$ is ample. Thus by adjunction
we have $C^2=-2-C\cdot K_X \geq -1$.
In addition, $W(X)$ is finite (of order $2^{10}3^{4}5^17$; see \cite[26.6]{refM}), and so is the set of
classes $[C]$ of rational curves $C$ with $C^2=-1$ (i.e., the classes of exceptional curves).
In fact, there are 56 of them (giving the 4 classes listed up to permutations
in the statement of the theorem),and their classes are precisely
the orbit of $E_7$ (see \cite[Proposition 26.1, Theorem 26.2(iii)]{refM}).
We note that these all are Ascenzi.

Now say $C^2>-1$; then $C$ is nef, hence there is an element $w\in W(X)$ such that $D=w[C]$
is a non-negative integer linear combination of the classes of
$H_0=L$, $H_1=L-E_1$, $H_2=2L-E_1-E_2$, and $H_i=3L-E_1-\cdots-E_i$ for
$3\leq i\leq 7$ \cite[Lemma 1.4, Corollary 3.2]{refH}.
Note that $H_7=-K_X$.

Write $D=[\sum_ia_iH_i]$. If $a_j>0$ for some $j\geq 3$, then we have
$-D\cdot K_X\leq D\cdot H_j \leq D\cdot\sum_ia_iH_i=D^2$, which violates
$D^2=-2-D\cdot K_X$. Thus $D=[a_0H_0+a_1H_1+a_2H_2]$.
If $a_0$ and $a_2$ are both positive, then
we get another violation, $-D\cdot K_X\leq D\cdot (H_0+H_2) \leq D^2$,
so either $a_0=0$ or $a_2=0$. If $a_0=0$, we cannot have $a_2\geq 2$, since
then $D^2=a_1a_2+a_2(2a_2+a_1)\geq 2(2a_2+a_1)=-D\cdot K_X$.
Likewise, if $a_2=0$, we cannot have $a_0>2$ nor can we have $a_0\geq2$ if $a_1\geq1$.
All that is left are the classes of $H_0+kH_1$, $2H_0$, $H_1$ and $H_2+kH_1$ for $k\geq0$,
all of which it is easy to see are classes of smooth rational curves. For example,
$H_0+kH_1$ corresponds to a plane curve of degree $k+1$ with a singular point of multiplicity $k$.
To find the numerical types $(d_C,m_1,\ldots,m_7)$ of all smooth rational $C$, it is now enough to
compute the orbit of each $D$ under $W(X)$, as we have done to produce the lists in the
statement of the theorem.

We now determine the splitting gaps for the non-Asscenzi cases.
First consider the curve $C$ of type $(10,4,4,4,4,4,4)$.
Since $a_C\leq d_C/2 =5$, it suffices to show that $a_C>4$ to prove that the gap is 0.
By twisting \eqref{eqnstar} by $\OO_C(C)$, we obtain an exact sequence
$$0\to \OO_C(4-a_C)\oplus\OO_C(4-b_C)\to \OO_C(C)\otimes H^0(L) \to \OO_C(C+L) \to0.$$
To show $a_C>4$ it now suffices to show this is exact on global sections,
since $h^0( \OO_C(C)\otimes H^0(L)) = 15 = h^0( \OO_C(C+L))$.
But exactness follows from
the fact that $H^0(C)\otimes H^0(L)\to H^0(C+L)$ is surjective (see \cite{refH2})
by taking global sections of the following diagram
$$\begin{matrix}
0 & \to & \OO_X\otimes H^0(L)  & \to & \OO_X(C)\otimes H^0(L)
& \to & \OO_C(C)\otimes H^0(L) & \to & 0 \cr
{} & {} & \downarrow  & {} & \downarrow & {} & \downarrow & {} & {} \cr
0 & \to & \OO_X(L)  & \to & \OO_X(C+L) & \to & \OO_C(C+L) & \to & 0 \cr
\end{matrix}$$
and applying the snake lemma.

We now find $a_C$ for the remaining non-Ascenzi cases.
By applying Lemma \ref{splitlem}, and Proposition \ref{unbalsplitting} with $A=3L-E_1-\cdots-E_7$,
we have $m_C\leq a_C\leq C\cdot A$, and except for
$(5, 2, 2, 2, 2, 2, 2, 0) + k(5, 2, 2, 2, 2, 2, 2, 1)$, in each case we have
$m_C=C\cdot A$, so $a_C=m_C$.

Finally, we consider the case $(5, 2, 2, 2, 2, 2, 2, 0) + k(5, 2, 2, 2, 2, 2, 2, 1)$ for $k>0$.
The preceding argument shows only that $2+2k\leq a_C\leq 3+2k$, but in fact,
$a_C=3+2k$ for $k>0$ and $a_C=2$ for $k=0$ (hence the splitting gap is $|k-1|$).
Certainly $a_C=5$ if $k=1$, since a curve of type $(10, 4, 4, 4, 4, 4, 4, 1)$
is a proper transform of, but isomorphic to, a curve of type $(10, 4, 4, 4, 4, 4, 4, 0)$
and thus has the same splitting type.

To see $a_C=3+2k$ when $k>1$, let $F$ and $G$ be smooth rational curves with
$[F]=[5L-2(E_1+\cdots+E_6)]$ and let $[G]=[5L-2(E_1+\cdots+E_6)-E_7]$.
Thus $[C]=[F+kG]$; note also that $2F-C=F-kG$. Taking cohomology of the diagram
$$\begin{matrix}
0 & \to & \OO_X(2F-C)\otimes H^0(L)  & \to & \OO_X(2F)\otimes H^0(L)
& \to & \OO_C(2k+2)\otimes H^0(L) & \to & 0 \cr
{} & {} & \downarrow  & {} & \downarrow & {} & \downarrow & {} & {} \cr
0 & \to & \OO_X(L+2F-C)  & \to & \OO_X(2F+L) & \to & \OO_C(7k+7) & \to & 0 \cr
\end{matrix}$$
gives the following commutative diagram:
$$\begin{matrix}
0 & \to & H^0(2F)\otimes H^0(L) & \to & H^0(\OO_C(2k+2))\otimes H^0(L) & \to & H^1(2F-C)\otimes H^0(L)  & \to & 0 \cr
{} & {\vbox to.2in{\vss}} & \downarrow \mu_2  & {\lower.15in\vbox to.15in{\vss}} & \downarrow \mu_3 & {} & \downarrow \mu_1 & {} & {} \cr
0 & \to  & H^0(2F+L) & \to & H^0(\OO_C(7k+7)) & \to & H^1(L+2F-C)  & \to & 0 \cr
\end{matrix}$$

For each $i$, let $V_i={\rm ker}(\mu_i)$.
The results of \cite{refH2} show that $V_2=0$. If we also show that $V_1=0$, then the snake lemma
shows that $H^0(\OO_C(2k+2-a_C))\oplus H^0(\OO_C(2k+2-b_C))=V_3=0$, and thus that
$a_C>2k+2$, so $a_C=2k+3$.

To justify that $V_1=0$, consider
$$\begin{matrix}
0 & \to & \OO_X(F-(d+1)G)\otimes H^0(L)  & \to & \OO_X(F-dG)\otimes H^0(L)
& \to & \OO_G(1)\otimes H^0(L) & \to & 0 \cr
{} & {} & \downarrow  & {} & \downarrow & {} & \downarrow & {} & {} \cr
0 & \to & \OO_X(L+F-(k+1)G)  & \to & \OO_X(L+F-kG) & \to & \OO_G(6) & \to & 0. \cr
\end{matrix}$$
We know $h^0(F-kG)=h^0(L+F-kG)=0$ for all $k\geq2$
since $(F-kG)\cdot L<0$ and $(L+F-kG)\cdot L<0$. Also,
$h^2(F-kG)=h^2(L+F-kG)=0$ for all $k\geq2$
by duality, since $G$ is nef
and $G\cdot(K_X-(F-kG)))<0$ and $G\cdot(K_X-(L+F-kG))<0$.
Thus we can use Riemann-Roch to obtain $h^1(F-kG) = 2k-3$ and
$h^1(L+F-kG) = 7k-10$ when $k\geq 2$.

When $k=1$, $h^1(F-kG)=h^1(E_7)=0$
and $h^1(L+F-kG)=h^1(L+E_7)=0$.
Taking cohomology when $k=1$ now gives a commutative
diagram with exact rows:
{\small
$$\begin{matrix}
0 & \to & H^0(F-G)\otimes H^0(L) & \to & H^0(\OO_G(1))\otimes H^0(L) & \to & H^1(F-2G)\otimes H^0(L)  & \to & 0 \cr
{} & {} & \downarrow  & {} & \downarrow & {} & \downarrow & {} & {} \cr
0 & \to & H^0(L+F-G) & \to & H^0(\OO_G(6)) & \to & H^1(L+F-2G)  & \to & 0. \cr
\end{matrix}$$}
The left vertical map is an isomorphism and the middle vertical map is injective (since
the splitting type of $G$ is $(2,3)$), so the snake lemma tells us that
the right vertical map is injective.

Now take cohomology again but with some $k\geq 2$. We obtain another commutative
diagram with exact rows:
{\small
$$\begin{matrix}
0 & \to & H^0(\OO_G(1))\otimes H^0(L) & \to & H^1(F-(k+1)G)\otimes H^0(L)  & \to & H^1(F-kG)\otimes H^0(L) & \to & 0 \cr
{} & {} & \downarrow  & {} & \downarrow & {} & \downarrow & {} & {} \cr
0 & \to & H^0(\OO_G(6)) & \to & H^1(L+F-(k+1)G)  & \to & H^1(L+F-kG) & \to & 0. \cr
\end{matrix}$$}
By induction the right vertical map is injective and we saw above that the left one is also injective,
hence so is the middle one; i.e., $V_1=0$ for all $k\geq0$.
\end{proof}

\begin{proof}[Proof of Theorem \ref{7ptthm}]
(a) By inspection of the statement of Theorem \ref{list7},
we see that in order for $C$ to fail to be Ascenzi, its image in
the plane must have at least 6 singular points, and we see that there is a unique
numerical type with exactly 6, namely $(10, 4, 4, 4, 4, 4, 4)$, and its splitting gap is 0.

(b) This follows from inspection of the statement of Theorem \ref{list7}.
\end{proof}

\subsection{Exceptional curves on 9 point blow ups}\label{9pts}

We would like to apply our results to the case of blow ups $X$ of $\pr2$ at $r=9$ generic points.
To do so it will be helpful to collect some facts about the exceptional divisors
on such an $X$.

As mentioned in the introduction, the case $r=9$ is the first interesting case for the problem of splitting types
of exceptional curves on blow ups of $\pr2$ at $r$ generic points, since the exceptional curves
have only finitely many numerical types when $r<9$. The numerical types for $r<8$ are obtained by deleting 0 entries
from those for $r=8$ so it's enough to list the types for $r=8$. Up to permutations of the entries $m_i$,
the types for $r=8$ are as follows \cite{refM}:
$(0,0,\ldots,0,-1)$,
$(1,1,1,0,\ldots,0)$,
$(2,1,1,1,1,1,0,\ldots,0)$,
$(3,2,1,1,1,1,1,1,0,\ldots,0)$,
$(4,2,2,2,1,1,1,1,1)$,
$(5,2,2,2,2,2,2,1,1)$, and
$(6,3,2,2,2,2,2,2,2)$.
As is evident, these all are Ascenzi.

For $r=9$ it is well known that there are infinitely many numerical types of exceptional curves \cite{refN}.
The recognition that only finitely many of them are Ascenzi seems to be new.
To proceed to justify both of these facts, we begin by recalling
how to write down the numerical types of exceptional curves
for $r=9$. The result is old enough to be hard to attribute, especially in the form
we will need, so for the convenience of the reader we include a proof.

\begin{prop}\label{enumexc}
Let $X\to \pr2$ be obtained by blowing up $r=9$ distinct points $p_i$,
with $L,E_1,\ldots,E_9$ the usual basis of the divisor class group ${\rm Cl}(X)$
with respect to this blow up.
\begin{itemize}
\item[(a)] A class $[E]\in {\rm Cl}(X)$ satisfies $E^2=E\cdot K_X=-1$ if and only if
$[E]=v+(v^2/2)[K_X]+[E_9]$ for an element $v\in{\rm Cl}(X)$ with $v\cdot K_X=v\cdot E_9=0$.
Moreover, the element $v$ is unique.
\item[(b)] Assume the points $p_i$ are generic.
Then a class $[E]\in {\rm Cl}(X)$ satisfies $E^2=E\cdot K_X=-1$ if and only if $[E]$ is the
class of an exceptional curve. Thus the classes of exceptional curves
are exactly the classes of the form $v+v^2[K_X]/2+[E_9]$ for $v\in K_X^\perp\cap E_9^\perp$.
\end{itemize}
\end{prop}

\begin{proof} (a)
If $[E]=v+(v^2/2)[K_X]+[E_9]$ where $v\cdot K_X=v\cdot E_9=0$, then it is just a calculation
to check that $E^2=E\cdot K_X=-1$. Conversely,
if $E^2=E\cdot K_X=-1$, then $(E-E_9)\cdot K_X=0$. But $K_X^\perp$ is negative semi-definite
and even (i.e., $v\in K_X^\perp$ implies $2 | v^2\leq 0$)
with the only elements $v\in K_X^\perp$ having $v^2=0$ being
multiples of $[K_X]=[-3L+E_1+\cdots+E_9]$. If $r=(E-E_9)\cdot E_9$, then
$[(E-E_9)+rK_X]$ is in $K_X^\perp\cap E_9^\perp$, which is known to
be negative definite, spanned by the classes of the elements
$r_0=L-E_1-E_2-E_3, r_1=E_1-E_2,\ldots, r_7=E_7-E_8$. If we set $v=[(E-E_9)+rK_X]$,
we obtain $[E]=v-r[K_X]+[E_9]$ and now using the fact that $E^2=E\cdot K_X=-1$,
we find that $v^2=-2r$, hence $[E]=v+v^2[K_X]/2+[E_9]$.
To see uniqueness, assume $v+(v^2/2)[K_X]+[E_9]=w+(w^2/2)[K_X]+[E_9]$.
Then $v+(v^2/2)[K_X]=w+(w^2/2)[K_X]$, so $v^2/2=-E_9\cdot (v+(v^2/2)K_X)=
-E_9\cdot (w+(w^2/2)K_X)=w^2/2$, so $(v^2/2)[K_X]=(w^2/2)[K_X]$ and hence $v=w$.

(b) To prove the backward implication, note that, by adjunction,
if $E$ is an exceptional curve, then $E^2=E\cdot K_X=-1$.
Conversely, assume $E^2=E\cdot K_X=-1$.
Since $X$ is obtained by blowing up generic points, $[-K_X]$ is the class of a
reduced and irreducible curve $\Gamma$ with $-K_X^2=0$, and moreover there
are no smooth rational curves $C$ with $-K_X\cdot C=0$; such a curve $C$ must have $C^2=0$, but there are no such $(-2)$-curves,
since $[C]$ would reduce by a Cremona transformation
centered in the points $p_i$ to $[L-E_1-E_2-E_3]$ (see \cite[\S0]{refH4}, \cite{refK}),
but the images $p_i'$ of the points $p_i$ under the transformation
are generic \cite[proof of Lemm 2.5]{refN} so no three of the points $p_i'$ can lie on a line.
Since $\Gamma \cdot E=1$ and there are no $(-2)$-curves, it follows by \cite[Proposition 3.3]{refLH} that
$E$ is an exceptional curve.
\end{proof}

\begin{rem}
A class $E$ with $E^2=E\cdot K_X=-1$ need not be the class of an
exceptional curve when $r>9$; for example,
$[K_X]$ is such a class when $r=10$, but since $L$ is nef and $L\cdot K_X=-3$,
$[K_X]$ is not the class of an effective divisor.
\end{rem}

We now show for $r=9$ that there are only finitely many exceptional curves $E$
satisfying the condition $d_E\leq 2m_E+1$ and hence there are only finitely many
Ascenzi exceptional curves when $r=9$. In fact, we show more:

\begin{prop}\label{finAscprop}
Let $X\to \pr2$ be obtained by blowing up $r=9$ distinct points $p_i$.
Then for each integer $j$ there are only finitely many classes $E$ of exceptional curves
such that $d_E-2m_E\leq j$.
\end{prop}

\begin{proof}
Let $L,E_1,\ldots,E_9$ be the basis of the divisor class group ${\rm Cl}(X)$
with respect to the blow up $X\to \pr2$.
Since $E$ is effective and $L$ is nef, we have $d_E=E\cdot L\geq 0$.
Moreover, since $d_E-2m_E=E\cdot(L-2E_i)$ for some $i$,
it is enough to show for each $i$ that there are only finitely many $E$ such that $E\cdot(L-2E_i)\leq j$.
The proof is the same for each $i$; we will thus consider the case $i=1$.
Since any exceptional curve $C$ satisfies $C^2=C\cdot K_X=-1$, it is enough
to show that there are only finitely many classes $E$ (whether or not they are classes of exceptional curves)
with $E^2=E\cdot K_X=-1$ such that $E\cdot(L-2E_1)\leq j$ and such that $d_E\geq0$.
If we find an upper bound on $d_E$, depending only on $j$, we will be done. To obtain it, note
by Proposition \ref{enumexc}(a) that we
have $E=v+v^2[K_X]/2+[E_9]$ for some
$v=[a_0r_0+a_1r_1+\cdots+a_7r_7]=[a_0L-(a_0-a_1)E_1-b_2E_2-\cdots-b_8E_8]
\in K_X^\perp\cap E_9^\perp$.
Hence $E\cdot(L-2E_1)=-a_0+2a_1-v^2/2$, and, since $v\cdot K_X=0$,
$2a_0+a_1=b_2+\cdots+b_8$. Thus the average $\bar{b}=(b_2+\cdots+b_8)/7$ is
$(2a_0+a_1)/7$. Working formally over the rationals, let
$w=[a_0L-(a_0-a_1)E_1-\bar{b}(E_2+\cdots+E_8)]$, so $w\in K_X^\perp\cap E_9^\perp$
and $w^2/2=(5a_0a_1-4a_1^2-2a_0^2)/7$. Due to the fact that the intersection form is negative semi-definite on $K_X^\perp$
and the general fact for averages that the square of an average is at most the average of the squares
and hence $7\bar{b}^2\leq b_2^2+\cdots+b_8^2$, we have $0\leq -w^2/2\leq -v^2/2$.
Thus $E\cdot(L-2E_1)=-a_0+2a_1-v^2/2\geq -a_0+2a_1-w^2/2=(2a_0^2+4a_1^2-5a_0a_1-7a_0+14a_1)/7$.
The substitution $a_0=x+5y-2$ and $a_1=4y-3$ gives
$(2a_0^2+4a_1^2-5a_0a_1-7a_0+14a_1)/7=(2x^2+14y^2-14)/7$.

Since $d_E=E\cdot L = a_0-3v^2/2$, we have $j\geq E\cdot (L-2E_1)=-a_0+2a_1-v^2/2=d_E/3-4a_0/3+2a_1$.
Using the substitution $a_0=x+5y-2$ and $a_1=4y-3$ and simplifying gives
$d_E\leq 3j+4(x-y)+10$, where $j\geq (2x^2+14y^2-14)/7$.
Using Lagrange multipliers, we see that the maximum value of $x-y$
given $j\geq (2x^2+14y^2-14)/7$ occurs
for $x=\lambda/4$ and $y=-\lambda/28$ when $j= (2x^2+14y^2-14)/7$, hence
$$d_E\leq 3j+\Big\lfloor 4\sqrt{4j+8}\Big\rfloor+10.
\eqno{(\circ)}$$
Clearly there are only finitely many classes
$E=d_EL-m_1E_1-\cdots-m_9E_9$ with $d_E\geq0$ and $E^2=-1$ satisfying $(\circ)$.
\end{proof}

\begin{cor}\label{BoundOnDegree} Let  $X$ be the blow up of $\pr2$ at 9 distinct points. Then every Ascenzi exceptional curve $E\subset X$ has $d_E\leq 26$ and the only one with unbalanced splitting is $E=4L-3E_1-E_2-\cdots-E_9$ (up to indexation of the $E_i$).
\end{cor}

\begin{proof} The Ascenzi exceptional curves $E$ satisfy $d_E-2m_E\leq 1$.
If we set $j=1$ in $(\circ)$, then $d_E\leq 26$; i.e., on a blow up $X$ of $\pr2$ at 9 points
every Ascenzi exceptional curve $E$ must have $d_E\leq 26$.
In order for an Ascenzi exceptional curve $E$ to be unbalanced, we must have
$m_E-(d_E-m_E)\geq 2$; i.e., we must have $d_E-2m_E\leq -2$.
But in the notation of the proof of  Proposition \ref{finAscprop}, $d_E-2m_E\leq -2$ implies
$-2\geq E\cdot(L-2E_1)\geq (2x^2+14y^2-14)/7$, which forces
$x=y=0$, hence $-2= E\cdot(L-2E_1)=(2x^2+14y^2-14)/7$ and thus $|\bar{b}|=|b_2|=\cdots=|b_8|$.
But $x=y=0$ gives $a_0=-2$, $a_1=-3$ and $\bar{b}=-1$,
so $E=4L-3E_1\pm E_2\pm \cdots\pm E_9$, and now $E\cdot K_X=-1$
forces $E=4L-3E_1- E_2- \cdots- E_9$.
\end{proof}

\begin{rem}\label{AscenziList} Given fixed integers $d>0$ and $r>0$,
it is not hard using the action of $W(X)$ to find all classes $[E]=[d_EL-m_1E_1-\cdots-m_rE_r]$
of  exceptional curves satisfying $d_E\leq d$, where for efficiency it is best to require $m_1\geq \cdots\geq m_r$.
The method uses the fact that one can reduce any exceptional class $[E]$ to some $[E_i]$ by
successively applying quadratic transforms $s_{ijk}$, centered at
$E_i$, $E_j$ and $E_k$, choosing $i,j$ and $k$ so that $d_E$
drops as much as possible each time (just choose $i,j,k$ to maximize the sum $m_i+m_j+m_k$).
Applying this in reverse, one starts with $[E_1]$ and applies $s_{ijk}$ for various choices of $i,j,k$.
One continues doing this to the new classes one obtains; eventually one will have a list of classes $[E]$
with $d_E\leq d$ such that whenever one applies $s_{ijk}$
for any choice of $i,j,k$ to any $[E]$ on the list one always obtains (up to permutations of the $m_i$)
another $[E]$ on the list or an $E$ with $d>dE$. The list then is complete.

By using such an exhaustive procedure, we have found all $[E]$ with $d_E\leq 61$
for a blow up of $\pr2$ at 9 generic points.
There are all together 1054 exceptional classes $[E]=[d_EL-m_1E_1-\cdots-m_9E_9]$
with $d_E\leq 61$ and $m_1\geq \cdots\geq m_9$. Of these, 42 are Ascenzi, as follows.
By Corollary \ref{BoundOnDegree}, there are no other Ascenzi exceptional curves for $r=9$.

There is only one Ascenzi $E$ with $d_E-2m_E\leq-2$. It's numerical type is:
\tt
\begin{verbatim}
   4 3 1 1 1 1 1 1 1 1
\end{verbatim}
\rm

\noindent Those Ascenzi $E$ with $d_E-2m_E=-1$ are:
\tt
\begin{verbatim}
   1 1 1                  7 4 3 2 2 2 2 2 2 1   11 6 4 4 3 3 3 3 3 3
   3 2 1 1 1 1 1 1        9 5 3 3 3 3 3 2 2 2   13 7 4 4 4 4 4 4 4 3
   5 3 2 2 2 1 1 1 1 1
\end{verbatim}
\rm

\noindent Those Ascenzi $E$ with $d_E-2m_E=0$ are:
\tt
\begin{verbatim}
   0 0 0 0 0 0 0 0 0 -1   8 4 3 3 3 3 2 2 2 1   14  7 5 5 5 4 4 4 4 3
   2 1 1 1 1 1            8 4 4 3 2 2 2 2 2 2   14  7 6 4 4 4 4 4 4 4
   4 2 2 2 1 1 1 1 1     10 5 4 4 3 3 3 3 2 2   16  8 6 5 5 5 5 5 4 4
   6 3 2 2 2 2 2 2 2     12 6 4 4 4 4 4 4 3 2   18  9 6 6 6 6 5 5 5 5
   6 3 3 2 2 2 2 1 1 1   12 6 5 4 4 4 3 3 3 3   20 10 7 6 6 6 6 6 6 6
\end{verbatim}
\rm

\noindent And those Ascenzi $E$ with $d_E-2m_E=1$ are: \vskip0in
\tt
\begin{verbatim}
   5 2 2 2 2 2 2 1 1     13 6 5 5 5 4 4 3 3 3   19  9 7 6 6 6 6 6 6 4
   7 3 3 3 3 2 2 2 1 1   13 6 6 4 4 4 4 4 3 3   19  9 7 7 6 6 6 5 5 5
   9 4 4 3 3 3 3 3 2 1   15 7 6 5 5 5 5 4 4 3   21 10 7 7 7 7 7 6 6 5
   9 4 4 4 3 3 2 2 2 2   15 7 6 6 5 4 4 4 4 4   21 10 8 7 7 6 6 6 6 6
  11 5 4 4 4 4 4 3 2 2   17 8 6 6 6 6 5 5 4 4   23 11 8 8 7 7 7 7 7 6
  11 5 5 4 4 3 3 3 3 2   17 8 7 6 5 5 5 5 5 4   25 12 8 8 8 8 8 8 7 7
  13 6 5 5 4 4 4 4 4 2
\end{verbatim}
\rm
\end{rem}

In order to demonstrate that there are infinitely many non-Ascenzi exceptional curves
on a blow up of $\pr2$ at $r=9$ generic points, it will be useful first to prove
two lemmas.

\begin{lem}\label{usefullemma2}
Let $X$ be the blow up of $\pr2$ at $r=9$ generic points.
Let $E$ be an exceptional curve for which there is a divisor $A$
such that $[2A]=[E+K_X+L]$. Then $E$ has unbalanced splitting.
\end{lem}

\begin{proof}
We easily check that $-K_X\cdot A=2$.
If $E\cdot L=0$, then $E=E_i$ for some $i$ and there is no $A$
such that $[2A]=[E+K_X+L]$. Thus we may assume that $E\cdot L>0$, and we now
have $1+A^2=A\cdot L=L\cdot E/2-1\geq 0$ since $L\cdot E$ is even and positive.
Since $A\cdot L\geq0$ we know $h^2(A)=0$.
Now from Riemann-Roch we have $h^0(A)\geq (A^2-K_X\cdot A)/2+1=(1/2)(L\cdot E/2-2-K_X\cdot A)+1
=(L\cdot E)/4+1>0$.

By \cite[Lemma 4.1]{refLH}, the class of every effective divisor
is a non-negative sum of $[-K_X]$
and prime divisors of negative self-intersection. Since $X$ is a generic blow up, the only
prime divisors of negative self-intersection are the exceptional curves \cite{refH}.
But $E\cdot L\geq2$ so $2A\cdot E=-2+L\cdot E\geq0$, and for any exceptional curve $C\neq E$ we have
$2A\cdot C=(E+K_X+L)\cdot C\geq C\cdot K_X=-1$. Since $2A\cdot C$ is even we must have
$2A\cdot C\geq0$. Since $A$ is effective and meets $-K_X$ and every exceptional curve non-negatively,
$A$ is nef, but now $-K_X\cdot A>0$ implies $h^1(A)=0$ by \cite{refH}.

We now have ${\rm le}(A)\geq1$ by Lemma \ref{lelemma}, and since
$(A-E+L)\cdot L<0$, we have $h^0(A-C_A+L)=0$, so
$a_E\leq A\cdot E$ Proposition \ref{unbalsplitting}, hence $\gamma_E= d_E-2a_E\geq d_E-2A\cdot E=2$.
Thus $E$ has unbalanced splitting.
\end{proof}

\begin{cor}\label{usefullemma}
Let $X$ be the blow up of $\pr2$ at $r=9$ generic points.
Let $[E]=[dL-m_1E_1-\cdots-m_9E_9]$ be the class of an exceptional curve
with $m_1\geq\cdots\geq m_9\geq 0$ and $d\geq 2m_1-1$.
Let $A=E+E_1-sK_X$ for $s=d-2m_1+1$ and let $C_A=2A-K_X-L$.
Then $[C_A]$ is the class of an exceptional curve with unbalanced splitting.
\end{cor}

\begin{proof}
Direct calculation shows $C_A^2=K_X\cdot A=-1$, hence $[C_A]$ is the class of an exceptional curve by
Proposition \ref{enumexc}, and it has unbalanced splitting by Lemma \ref{usefullemma2}.
\end{proof}

\begin{proof}[Proof of Theorem \ref{classificationthm}]
Parts (a) and (b) follow from Corollary \ref{BoundOnDegree}.
Consider part (c). By Proposition \ref{enumexc}(b), there are infinitely many exceptional curves on $X$.
For any fixed $d$,  there can be at most finitely many classes $E=dL-m_1E_1-\cdots-m_9E_9$
with $E^2=-1$. Thus for any $d$, there are infinitely many exceptional curves $E$ with $E\cdot L\geq d$.
By Corollary \ref{BoundOnDegree}, for $d>26$
none of these infinitely many exceptional curves is Ascenzi, and hence for each such exceptional
curve $E$ we have $d_E> 2m_E+1$. For each such $E$ we thus have by Corollary \ref{usefullemma}
an unbalanced exceptional $C_A$ with $C_A\cdot L>E\cdot L\geq d$, and hence there are
infinitely many non-Ascenzi unbalanced exceptional curves.
\end{proof}

\begin{proof}[Proof of Theorem \ref{thm2}]
By Lemma \ref{usefullemma2}, $E$ has unbalanced splitting
since there is a divisor $A$ with $2A=E+K_X+L$, hence
$\gamma_E\geq 2$ and $a_E=(d_E-\gamma_E)/2\leq (d_E-2)/2$.
\end{proof}

\begin{rem}\label{infnonAscbal}
Let $X$ be the blow up of $r=9$ generic points of $\pr2$.
Here we explain why there are infinitely many non-Ascenzi exceptional
curves $E\subset X$ for which there is no divisor $A$ satisfying $2A=E+K_X+L$.
(Note if Conjecture \ref{9ptconj} is true, each such $E$ must have balanced splitting.)
We know $X$ has infinitely many classes
$[E']=[d_{E'}L-m_1E_1-\cdots-m_9E_9]$ of exceptional curves $E'$, and we may assume
$m_1\geq m_2\geq \cdots \geq m_9\geq 0$.
We have seen that only finitely many of them are Ascenzi.
As in the proof of Theorem \ref{classificationthm}, there are
infinitely many $E'$ such that $2A=E'+K_X+L$ for some $A$.
For each such $E'$, we thus see $d_{E'}$ is even and each $m_i$ is odd.
Note that $E'\cdot(L-E_7-E_8-E_9)>0$, because
$m_1\geq m_2\geq \cdots \geq m_9\geq 0$ implies
$E'\cdot(L-E_1-E_2-E_3)\leq E'\cdot(L-E_4-E_5-E_6)\leq E'\cdot(L-E_7-E_8-E_9)$, so if we had
$E'\cdot(L-E_7-E_8-E_9)\leq0$, we would have
$1=E'\cdot(-K_X)=E'\cdot((L-E_1-E_2-E_3)+(L-E_4-E_5-E_6)+(L-E_7-E_8-E_9))\leq0$.
But $[E]=s_{789}([E'])$ is the class of an exceptional curve $E$,
and $E'\cdot(L-E_7-E_8-E_9)>0$ implies that $d_E>d_{E'}$ where $d_E$ is odd.
\end{rem}

\begin{rem}\label{twoptsofview}
Let $X$ be the blow up of $\pr2$ at $r=9$ generic points.
By Conjecture \ref{9ptconj}, an exceptional curve $E$ on $X$ has unbalanced splitting
if and only if there is a certain divisor $A$ with $-K_X\cdot A=2$. In the conjecture,
$[A]$ has the form $[E+K_X+L]/2$, but in Corollary \ref{usefullemma},
$[A]$ has the form $[E'+E''-sK_X]$ where $E'\ne E''$ are exceptional curves and $s\geq0$.
However, as noted in the proof of Lemma \ref{usefullemma2},
the class of every effective divisor on $X$ is a non-negative sum of $[-K_X]$
and classes of prime divisors of negative self-intersection. Thus if $D$ is an effective divisor
with $-K_X\cdot D=d$, then we can write $[D]$ as a sum of classes of $d$ exceptional curves
plus some non-negative multiple of $[-K_X]$. In particular, if $[A]=[E+K_X+L]/2$,
then we also have $[A]=[E'+E''-sK_X]$ as above.

\end{rem}

\renewcommand{\thethm}{\thesection.\arabic{thm}}
\setcounter{thm}{0}

\section{Application to graded Betti numbers for fat points}\label{appls}

Let $p_1,\ldots,p_r\in\pr2$ be points. A 0-dimensional subscheme $Z\subset\pr2$
with support contained in the set of points $p_i$ is called a
\emph{fat point} subscheme if it is defined by a homogeneous ideal $I\subset
R = \KK[\pr2]$ of the form $I=\cap_i(I(p_i)^{m_i})$ where each $m_i$
is a non-negative integer. In this case we will write $I=I_Z$, and
 $Z=m_1p_1+\cdots+m_rp_r$. The least degree $t$ such that the
homogeneous component $I_t$ of $I$ of degree $t$ is non-zero is
denoted $\alpha(Z)$, or just $\alpha$ if $Z$ is understood.

We are interested in determining the minimal free graded resolution
for the ideal $I$ of a scheme of fat points in $\pr 2$; our aim,
following the work in \cite{refF1}, \cite{refFHH}, \cite{refGHI1},  \cite{refGHI2} and
\cite{refGHI3}, is to study the graded Betti numbers of $I$ when the
support of $Z$ is given by generic points in $\pr2$.

Notice that the values of the Hilbert function of $Z$, $h_Z(k) =
\dim R_k - \dim I_k$, are described, under the genericity
assumption, by a well known conjecture by means of which one can
explicitly write down the function $h_Z$ given the multiplicities
$m_i$. Various equivalent versions of this conjecture have been
given (see \cite{refS}, \cite{refH5}, \cite{refGi}, \cite{refHi}).
We will refer to them collectively as the SHGH Conjecture. Roughly,
the SHGH conjecture says that $h_Z(k)$ does not assume the expected
value if and only if the linear system $|I_k|$ presents a multiple
fixed rational component.

When trying to state a conjecture for the graded Betti numbers of
$I$, the situation turns out to be much more complicated. For
general simple points, it is known that the minimal resolution is
``as simple as it can be", i.e. for each $ k$, $\mu_k : I_k \otimes
R_1 \rightarrow I_{k+1}$ has maximal rank. So, the first problem that comes to mind is to understand in which
cases the resolution of $I_Z$ can be different from the resolution
of $l(Z)=length (Z)$ general simple points, which amounts to finding
the values $k$ for which $\mu_k$ does not have maximal rank.

Of course there are trivial cases with ``bad resolution", namely
those for which $Z$ has ``bad postulation".
Hence we are interested first in finding cases where $Z$ is supported on generic points and has
generic Hilbert function (assuming SHGH), but it has a ``bad
resolution". In those cases (e.g., see \cite[Remark 2.3]{refGHI2}) it is easy to
check that the only value of $k$ for which
$\mu_k$ might not have maximal rank is $k=\alpha$.

Our idea, consistent with the known examples, is that the
``troubles" are always given by the existence of rational curves
whose intersection with the fat point scheme $Z$ is too high with
respect to the behavior of the cotangent bundle on the curve, or,
to be more precise, to the splitting of the pull back of the
cotangent bundle on the normalization of the curve. In other
words, the scheme $Z$ has a ``too high secant" rational curve.
This is the analogue of what happens with curves in $\pr 3$,
where, for example, the generic rational quintic curve postulates
well but has a bad resolution, and this is due to the fact that
the quintic has a 4-secant line (see \cite{refGLP}).

For example, $Z=3p_1+3p_2+p_3+p_4+p_5$ should be generated by
quintics, but it is not since the line $L$ through $p_1$ and $p_2$
is a fixed component for the quintics. Another way to look at this
is that the intersection of $Z$ with $L$ is a scheme of length 6,
while $\Omega (6) \vert_L \cong \OO_{L}(4)\oplus\OO_{L}(5)$, so
that its sections vanishing on $Z$ also vanish along $L$; i.e.,
$Z\cap L$ imposes independent conditions on one direct summand,
but not on $\OO_{L}(4)$, with the result that the cokernel of
$H^0(\Omega (6) \vert_L) \to H^0(\Omega (6)\vert_Z)$ is non-zero.
But this cokernel is the surjective image of the cokernel of
$\mu_5(Z)$, and hence $\mu_5(Z)$ cannot be surjective (for a
detailed explanation, see \cite{refGHI2}, especially the
commutative diagram in the proof of \cite[Proposition
4.2]{refGHI2}, analogous to \eqref{diagram} below).

Other plane curves $C$ can play the role of $L$, but understanding $\Omega(k+1)\vert_C$ is
more difficult when $C$ is not a smooth rational curve, because when $C$ is not smooth
and rational, $\Omega\vert_C$ need not split. One way to deal with this is to look at
$(\pi^*\Omega (k+1))\vert_{C'}$ for smooth rational curves $C'\subset X$,
where $\pi:X\to\pr2$ is the blow up of points $p_i$ (and
hence typically $C'$ is the normalization of some plane curve $C$).
The forms in $I_k$ will correspond to divisors in the class of $F_k = kL-m_1E_1-\dots
-m_rE_r$. In order to study the maps $\mu_k$, we will, equivalently,
consider the maps $\mu_{F_k}: H^0(F_k)\otimes H^0(L) \rightarrow
H^0(F_{k}+L)$; since we are interested in the case $k=\alpha$, we
set $F=F_\alpha$.

So consider a rational curve $C \subset \pr 2$ whose strict transform
$C' \subset X$ is smooth  and irreducible; setting $t=F\cdot C'$, $a=a_{C'}$, $b=b_{C'}$,
via twisting the sequence \eqref{eqnstar} by $F$ we get:
\begin{equation}\label{eqn4b}
0 \to {\OO}_{C'}(t-a)\oplus{\OO}_{C'}(t-b) \to
F|_{C'}\otimes H^0(L) \to (F+L)|_{C'} \to 0.
\end{equation}

Taking cohomology, we get the map $\bar \mu_{C',F}:
H^0(F\vert_{C'})\otimes H^0(L)  \to H^0((F+L)\vert_{C'})$ where
$\hbox{ker}(\bar \mu _{C',F})=H^0({\OO}_{C'}(t-a)\oplus
{\OO}_{C'}(t-b))$.

Assuming $H^1(F-C')=0$ and $L\cdot (F-C')\geq-1$, which imply $H^1(F-C'+L)=0$, we have (as
in \cite{refGHI2}) the following commutative diagram:

\begin{equation}\label{diagram}
\begin{matrix} {} & {} & 0 & {} & 0 & {} & 0 & {} & {} \cr
{} & {} & \downarrow & {} & \downarrow & {} & \downarrow & {} & {}
\cr 0 & \to & H^0((F-C')\otimes p^*\Omega (1)) & \to & H^0(F\otimes
p^*\Omega (1)) & \to & \hbox{ker}(\bar \mu _{C',F}) &
\xrightarrow{\tau} & \cr {} & {} & \downarrow & {} & \downarrow & {} &
\downarrow & {} & {} \cr 0 & \to & H^0(F-C')\otimes H^0(L)  & \to &
H^0(F)\otimes H^0(L) & \to & H^0(F\vert_{C'})\otimes H^0(L) & \to &
0 \cr {} & {} & \downarrow \raise3pt\hbox to0in{$\scriptstyle\mu_{F-
C'}$\hss} & {} & \downarrow\raise3pt\hbox
to0in{$\scriptstyle\mu_{F}$\hss} & {} & \downarrow\raise3pt\hbox
to0in{$\scriptstyle\bar \mu _{C',F}$\hss} & {} & {} \cr 0 & \to &
H^0(F-C'+L) & \to & H^0(F+L) & \to & H^0((F+L)\vert_{C'})& \to & 0
\cr {} & {} & \downarrow & {} & \downarrow & {} & \downarrow & {} &
{} \cr {} & \xrightarrow{\tau} & \hbox{cok}\mu_{F-{C'}} &
\to & \hbox{cok}\mu_{F} & \to & \hbox{cok}(\bar \mu _{C',F}) & \to &
0 \cr {} & {} & \downarrow & {} & \downarrow & {} & \downarrow & {}
& {} \cr {} & {} & 0 & {} & 0 & {} & 0 & {} & {} \cr
\end{matrix}
\end{equation}

If $C'$ also satisfies $t=F\cdot C' \geq -1$, then
$H^1(F\vert_{ C'})\otimes H^0(L) =0$
so the last vertical column of \eqref{diagram} gives cok$(\bar \mu _{C',F})=H^1({\OO}_{
C'}(t-a)\oplus {\OO}_{C'}(t-b))$. In this case, $\mu_{F}$ cannot
have maximal rank if cok$(\bar \mu _{C',F})$ is ``too big" (when
$\mu_{F}$ is expected to be surjective, too big means simply that
cok$(\bar \mu _{C',F})$ is nonzero).
We will now see how this all works with two examples which use rational
curves $C\subset \pr 2$ with unbalanced splitting.

\begin{example}\label{ex4111} Let $Z= 4p_1+p_2+\cdots+p_9$. It is
well known that $Z$ has good postulation; we have $l(Z)=18$, $\dim
(I_Z)_{4}=0$, $\dim (I_Z)_{5}=3$ so $\alpha(Z)
=5$, and $\dim (I_Z)_{6}=10$, hence one expects that $\mu_{5}$ is injective and that $\dim
{\rm coker}\mu_{5}=1$. We will see that this does not happen.
Consider a quartic curve $C\subset \pr 2$ passing through the
$p_i$'s and with a singularity of multiplicity 3 at $p_1$.
Its strict transform is a divisor $C'=4L-3E_1-E_2-\dots -E_9$ on
$X$; $C'$ is Ascenzi with unbalanced splitting
$(a_{C'},b_{C'})=(1,3)$. If we consider the diagram (\ref{diagram})
where $F=F_{5}$ and $t=F\cdot C'=20-20=0$, we get $\dim {\rm
cok}(\bar \mu _{C',F})=h^1({\OO}_{C'}(-1)\oplus {\OO}_{C'}(-3))=2$.
This forces $\dim {\rm coker}\mu_{F}\geq 2$, and we actually have
$\dim {\rm coker}\mu_{F}= 2$, since $F-C'=L-E_1$, so the column on
the left column of the diagram corresponds to the linear syzygies on the pencil of
lines through the point $p_1$, but in that case we know cok$( \mu_{F-C'})=0$.
\end{example}

\begin{example}\label{ex44411} Let $Z= 4p_1+\cdots+4p_7+p_8+p_9$;
we know that $Z$ has good
postulation and $(I_Z)_{11}$ is fixed component free (e.g. see
\cite{refH}). Namely, we have  $l(Z)=72$, $\dim (I_Z)_{10}=0$, $\dim
(I_Z)_{11}=6$, $\dim (I_Z)_{12}=19$, $\alpha(Z) =11$, hence
$\mu_{11}$ is expected to be injective, with  $\dim {\rm
cok}(\mu_{11})=1$, but we will see that this does not happen (see
also \cite{refFHH}). This is due to the existence of a curve
$C\subset \pr 2$ of degree 8 where $m(C)_{p_i}=3$ for $1\leq i \leq
7$, and where $p_8$, $p_9$ are simple points of $C$, which by Example \ref{AEpairs}
gives $C'=8L-3E_1-\dots -3E_7-E_8-E_9$ on $X$
having unbalanced splitting $(a_{C'},b_{C'})=(3,5)$. Now from
diagram (\ref{diagram}), with $F=F_{11}$ and $t=F\cdot
C'=88-86=2$, we get $\dim {\rm cok}(\bar \mu _{C',F})=h^1({\OO}_{
C'}(-1)\oplus {\OO}_{C'}(-3))=2$. We have $F-C'=-K_X$, so the column
on the left of the diagram corresponds to the liner syzygies among
forms of degree 3 in the  resolution of the ideal of seven points in
$\pr 2$ for which we know cok$( \mu_{F-C'})=0$.  This implies that we actually
have $\dim {\rm coker}\mu_{F}= 2$.
\end{example}

Examples \ref{ex4111} and \ref{ex44411} give particular instances of
infinitely many fat point subschemes $Z\subset\pr2$ with ``bad resolution'',
which we can obtain using the results of \S3:

\begin{prop}\label{9fatpts}
Consider the blow up $X$ of $\pr 2$ at 9 generic points
$p_1,\dots,p_9$. Let $C'$ be an exceptional divisor on $X$
of type $(d,m_1,\dots ,m_9)=(2d',2m_1'+1,\dots ,2m_9'+1)$ with $d'\geq2$ and
consider the fat point subscheme
$Z=(3m_1'+1)p_1+\dots +(3m_9'+1)p_9\subset \pr 2$. Then
\begin{enumerate}
\item[$\bullet$] $Z$ has maximal Hilbert function and $\alpha(Z)=3d'-1$;
\item[$\bullet$] $\mu_\alpha$ is expected to be injective with $\dim {\rm
coker}(\mu_\alpha)=1$; but in fact
\item[$\bullet$] $ \dim {\rm coker}(\mu_\alpha)\geq 2$.
\end{enumerate}
Hence $Z$ does not have generic resolution.
\end{prop}

\begin{proof} If $A$ is a divisor of type $(d'-1,m_1',\dots ,m_9')$, then
$[2A]=[C'+K_X+L]$, so by Lemma \ref{usefullemma2} (and its proof)
$h^0(A)>0$ and $C'$ has unbalanced splitting; in particular,
$\gamma_{C'}\geq 2$ and the splitting type of $C'$ is
$(a_{C'},b_{C'})$ with $a_{C'}\leq d'-1$.  Note that $F=C'+A$
has type $(3d'-1, 3m_1'+1,\dots , 3m_9'+1)$.
Since $F$ is the sum of two effective divisors,
$\alpha(Z)=3d'+1$ follows if we check that
$h^0(F-L)=0$. Consider the exact sequence:
$$
0\to \OO _X(A-L) \to \OO_X(F-L) \to \OO _X(F-L)|_{C'} \to 0.
$$
Since $-K_X$ is nef and $-K_X\cdot(2A-2L)=-K_X\cdot(C'+K_X-L)<0$,
we see that $h^0(\OO _X(2A-2L))=0$ and hence also $h^0(\OO _X(A-L))=0$.
Moreover, $(F-L)\cdot C'=-d'-2$, so $h^0(\OO _X(F-L)|_{C'})=0$,
hence $h^0(F-L)=0$ and so $\alpha(Z)=3d'+1$.

In order to prove that $Z$ has maximal Hilbert function we only have to show
that $h_Z(3d'+1)$ is maximal, i.e. that $h^0(F)$ has the expected
dimension (equivalently, that $h^1(F)=0)$. But
$-K_X\cdot F=3$, so by \cite{refH}, $h^1(F)=0$ if we show that
$F$ is nef. But as noted in the proof of Lemma \ref{usefullemma2},
the class of every effective divisor is a non-negative sum of exceptional classes
and non-negative multiples of $-K_X$. Thus $F$ is nef if $F\cdot E\geq0$
for every exceptional curve $E$, but $F=(3C'+K_X+L)/2$, so
$F\cdot C'=d'-2\geq0$, while $F\cdot E\geq \lceil(-1+E\cdot L)/2\rceil\geq 0$ if $E\neq C'$.
Since $F$ is nef, it follows that $h^1(L, F+L)=0$, and since also $h^1(F)=0$, it follows and
that $h^1(F+L)=0$. A straightforward  (but tedious) computation now shows that
$h^0(F+L)-3h^0(F)=1$, hence $\mu_F$ is expected to be injective with
$\dim {coker}(\mu_F)=1$, as claimed.

Arguing as we did for $F$, we also see that $A$ is nef, and since $-K_X\cdot A>0$,
we have $h^1(F-C')=h^1(A)=0$, so we can
apply diagram (\ref{diagram}) for our $F$ and $C'$. We have that
$t=F\cdot C'=d'-2\geq a_{C'}-1$, so we get that
$\dim {\rm coker}(\bar \mu _{C',F})= h^1(\OO _C(d'-2-a_{C'})\oplus \OO
_C(d'-2-2d'+a_{C'}))$,
and $a_{C'}-d'-2\leq -3$, so $\dim {\rm coker}(\bar \mu _{C',F})\geq 2$.
\end{proof}

\medskip

Not all examples of fat point subschemes with good
postulation and ``bad resolution" follow the pattern
illustrated above. In fact, a more complicated geometry is possible;
for example, the curve $C'$ may have many irreducible rational components
and need not even be reduced (see Examples 4.7, 6.3 in \cite{refGHI2}).
Other examples can be found in \cite{refGHI1} or
in \cite{refGHI2}, where there are also two conjectures which
describe completely what the situation could be.

Resolutions for subschemes $Z$ not possessing a maximal hilbert
function are also of interest. Things are more complicated in this situation,
but the ``unbalanced splitting" idea can still be useful. Actually, when $r=9$ and the points
$p_i$ are generic, then using \cite[Theorem 3.3(b)]{refGHI1} and
assuming Conjecture \ref{9ptconj} if need be, we can in every degree
$k$, except possibly degree $\alpha(Z)+1$, find the minimal number
of generators of $(I_Z)_k$, as we demonstrate in the next example.

\begin{example}\label{exbignumbers} Let $Z=230p_1+225p_2+\cdots+225p_8+95p_9$, for generic points
$p_i \in \pr 2$. The Hilbert function of the ideal $I_Z$ can be
found by computing $h^0( F_k)$, where
$F_k=kL-230E_1-225E_2-\cdots-225E_8-95E_9$. We have $h^0(F_k)=0$
for $k<645$, $h^0(F_{645})=71$, $h^0(F_{646})=528$, $h^0(
F_{647})=1176$, $h^0(F_k)=\binom{k+2}{2}-\deg
Z=\binom{k+2}{2}-209100$. We will compute the rank of each map
$\mu_k$, except for $k=645$.

To find the minimal number $\nu_{k+1}$ of
generators  in each degree $k+1$ we must find the dimension of the
cokernel of the usual maps $\mu_{F_k}:H^0(F_k)\otimes H^0(L)\to
H^0(F_{k+1})$. Clearly $\nu_{645}=h^0(F_{645})=71$. The same
algorithm that we use to compute $h^0(F_k)$ can be used to give a
Zariski decomposition of $F_k$. This is useful since if $F_k=H+N$
where $H$ is effective and $N$ is effective and fixed in $|F_k|$,
then $\nu_{k+1}=\dim {\rm coker}(\mu_H)+(h^0(
F_{k+1})-h^0(H+L))$, and we know $h^0(F_{k+1})$ and
$h^0(H+L)$. It is known that the dimension $\delta_H$ of the
kernel of $\mu_H$ has bounds $h^0(H-(L-E_1))\leq \delta_H\leq
h^0(H-(L-E_1))+h^0(H-E_1)$. Bounds on $\delta_H$ of course give
bounds on $\dim {\rm coker}(\mu_H)$. We find $N=20E$, where
$E=20L-7E_1-\cdots-7E_8-3E_9$ is an exceptional curve which by
Conjecture \ref{9ptconj} has splitting gap 2, and
$H=245L-90E_1-85E_2\cdots-85E_8-35E_9$ is nef and effective. We find
$0=h^0(H-(L-E_1))\leq \delta_k\leq
h^0(F_k-(L-E_1))+h^0(F_k-E_1)=1$, $h^0(F_{k+1})=528$, $h^0(
H+L)=318$ and $h^0(H)=h^0(F_k)=71$ for $k=645$, and hence
$315\leq \nu_{646}\leq 316$.

For $t=646$ we have $N=0$ and $H=F_k$. Doing the same calculation
with this new Zariski decomposition gives $0\leq \nu_{647}\leq 99$.
But in fact, using the splitting gap of 2 from above and
\cite[Theorem 3.3(b)]{refGHI1} we have $\dim{\rm
coker}(\mu_{F_{646}})=\dim{\rm coker}(\mu_{L+20E})=
\binom{11}{2}+\binom{9}{2}=91$. From the Hilbert function we see
that the regularity of $I_Z$ is 647, so $\nu_k=0$ for $t>647$. Given
the Hilbert function and numbers of generators of $I_Z$ we compute
all but one of the remaining graded Betti numbers: there are 286
syzygies in degree 647 and 190 in degree 648, but since we do not
know the number of minimal generators in degree 646 we also do not
know the number of syzygies. This example and others like it can be
run at:
\url{http://www.math.unl.edu/~bharbourne1/GHM/ResForFatPts.html}.
\end{example}

\medskip
\begin{rem}
Notice that if $\dim (I_Z)_\alpha\le2$, then we can find the minimal
number of generators of $I_Z$ also in degree $\alpha(Z)+1$. If
$\dim (I_Z)_\alpha=1$, then $(I_Z)_\alpha\otimes R_1\to
(I_Z)_{\alpha+1}$ is injective so the number of generators in degree
$\alpha(Z)+1$ is just $\dim (I_Z)_{\alpha+1}-3$, while if $\dim
(I_Z)_\alpha=2$ we can determine the  dimension of the kernel of
$(I_Z)_\alpha\otimes R_1\to (I_Z)_{\alpha+1}$ since $(I_Z)_\alpha$ is a
pencil; indeed, assuming $Z=\sum_im_ip_i$ with $m_1\geq \cdots\geq
m_r$, the dimension of the kernel, which is either 0 or 1, is $\dim
(I_{Z-p_1})_{t-1}$.
\end{rem}

\end{document}